\documentclass{amsart}

\usepackage{float, bbm, hyperref, algorithm, algorithmic, caption, subcaption, mathrsfs, color, gensymb, verbatim, cite, 
	 comment, mathtools, enumerate}
\usepackage[loadonly]{enumitem}
\usepackage[all]{xy}

\usepackage{pdflscape}
\usepackage{rotating}
\usepackage{graphicx}

\usepackage{amsmath}
\usepackage{amsfonts}
\usepackage{amssymb}
\usepackage{amsthm}
\usepackage{mathrsfs}
\usepackage{mathtext}
\usepackage{mathtools}
\usepackage[mathscr]{eucal}

\usepackage{xcolor}

\definecolor{pink}{HTML}{FF1493}
\usepackage[colorinlistoftodos]{todonotes}

\def\RA#1{\textcolor[HTML]{844D92}{#1}}

\theoremstyle{definition}

\theoremstyle{remark}

\theoremstyle{theorem}
\newtheorem{thm}{Theorem}
\numberwithin{thm}{section} 

\newtheorem{cor}[thm]{Corollary}
\newtheorem{lemma}[thm]{Lemma}

\newtheorem{prop}[thm]{Proposition}

\theoremstyle{definition}
\newtheorem{defn}[thm]{Definition}

\theoremstyle{definition}
\newtheorem{example}[thm]{Example}

\theoremstyle{plain}

\newtheorem*{rmk}{Remark}

\theoremstyle{definition}

\newtheorem{OProb}{Open Problem}
\numberwithin{OProb}{section} 

\def\N{{\mathbb{N}}}
\def\Z{{\mathbb{Z}}}

\def\H{{\mathbb{H}}}

\def\doug{{\partial}}
\def\m{{\mathfrak{m}}}
\def\k{{\Bbbk}}

\def\ab{\Z\text{-}\mathrm{mod}}
\def\im{\mathrm{im}}

\def\crco{critical and cocritical degrees }

\def\Rmod{$\mathcal{R}\text{-}\mathrm{mod}\!$ }

\def\Ab{$\mathscr{A}b$}
\def\CS{\mathcal{S}}
\def\Ck{\mathcal{S}_{\k}}

\def\ZZ{\mathcal{Z^*}}

\def\CK{\mathscr{K}}
\def\x{\mathfrak{X}}

\def\CR{\mathscr{C}(R)\! }
\def\CQ{\mathscr{C}(Q)\! }

\def\compK{\mathcal{K}}
\def\bfu{\textbf{u}}

\newcommand{\stMCM}[1]{\operatorname{\underline{\mathrm{MCM}}}(#1)}
\newcommand{\Ktac}[1]{\operatorname{{\mathscr{K}}_{tac}}(#1)}
\newcommand{\comp}[1]{\mathcal{#1}}
\newcommand{\map}[1]{\mathit{#1}}
\newcommand{\mincomp}[1]{\overline{\mathcal{#1}}}
\newcommand{\shift}[1]{{\Sigma{#1}}}
\newcommand{\shiftq}[2]{{\Sigma^{#1}{#2}}}
\newcommand{\crdeg}[3]{\mathrm{crdeg}^{#1}_{{#2}}{#3}}
\newcommand{\cocrdeg}[3]{\mathrm{cocrdeg}^{#1}_{{#2}}{#3}}

\newcommand{\minmod}[1]{\overline{#1}}
\newcommand{\minmap}[1]{\bar{#1}}
\newcommand\mydots{\hbox to 1em{$\cdot\hss\cdot\hss\cdot$}}

\newcommand{\res}[1]{\mathbb{#1}}
\newcommand{\rk}[1]{\mathrm{rk}(#1)}
\newcommand{\tri}[7]{{{#1}}\xrightarrow{#2}{{#3}}\xrightarrow{#4}{\comp{#5}}\xrightarrow{#6}{\shift{\comp{#7}}}}

\newcommand{\Syz}[2]{\Omega^{#1}{#2}}

\newcommand{\Hom}{\operatorname{Hom}\nolimits}
\newcommand{\HomC}[2]{\operatorname{Hom(\comp{#1},\comp{#2})}}
\renewcommand{\Im}{\operatorname{Im}\nolimits}

\newcommand{\Cone}[1]{{\operatornamewithlimits{\mathcal{M}}\map{(#1)}}}
\newcommand{\coker}{\operatorname{coker}\nolimits}
\newcommand{\pd}{\operatorname{pd}\nolimits}
\newcommand{\id}{\operatorname{id}\nolimits}

\newcommand{\codim}{\operatorname{codim}\nolimits}

\newcommand{\depth}{\operatorname{depth}\nolimits}

\newcommand{\Tor}{\operatorname{Tor}\nolimits}
\newcommand{\Ext}{\operatorname{Ext}\nolimits}
\newcommand{\Tateext}{\operatorname{\widehat{Ext}}\nolimits}

\newcommand{\Id}{\operatorname{Id}\nolimits}
\renewcommand{\H}{\operatorname{H}\nolimits}

\newcommand{\cx}{\operatorname{cx}\nolimits}

\newcommand{\End}{\operatorname{End}\nolimits}

\newcommand{\comments}[1]{}
\newcommand{\f}{{\mathbf{f}}}

\newcommand{\homeq}{\simeq}
\newcommand{\Csum}[2]{ \comp{#1} \oplus  \comp{#2}}

\newcommand{\dual}[1]{\comp{#1}^*}

\newcommand{\der}[1]{\mathcal{D}(#1)}

\newcommand{\abs}[1]{\mid #1 \mid}
\newcommand{\betti}[1]{\{b_n^R({#1})\}}

\newcommand{\soc}{\operatorname{Soc}\nolimits}
\newcommand{\cosoc}{\operatorname{Cosoc}\nolimits}

\newcommand{\cdepth}{\operatorname{codepth}\nolimits}

\newcommand{\cdiam}[2]{\operatorname{diam}_{#1}(\comp{#2})}
\newcommand{\mdiam}[2]{\operatorname{diam}_{#1}({#2})}

\newcommand{\cres}[3]{\comp{#1}\to\res{#2}\dhrightarrow{#3}}

\newcommand{\gExt}[3]{\mathrm{Ext}_{#1}^{\ast}({#2},{#3})}
\newcommand{\gTateext}[3]{\Tateext_{#1}^{\ast}({#2},{#3})}

\newlist{arrowlist}{itemize}{1}
\setlist[arrowlist]{label=$\rightsquigarrow$}

\newlist{Arrowlist}{itemize}{1}
\setlist[Arrowlist]{label=$\Rrightarrow$}

\newlist{dlist}{itemize}{1}
\setlist[dlist]{label=$\diamond$}

\newcommand\dhrightarrow{%
	\mathrel{\ooalign{$\rightarrow$\cr%
			$\mkern3.5mu\rightarrow$}}}

\title[The Critical and Cocritical Degrees of a Totally Acyclic Complex]{The Critical and Cocritical Degrees of a Totally Acyclic Complex over a Complete Intersection}

\author[R.\ Aduddell]{Rebekah J.\ Aduddell}
\address{}
\email{rebekah.aduddell@mavs.uta.edu}

\date{\today}

\begin{document}
	\begin{abstract}
			It is widely known that the minimal free resolution of a module over a complete intersection ring has nice patterns eventually arising in its Betti sequence. Avramov, Gasharov, and Peeva defined the notion of \emph{critical degree} for a finitely generated module $M$ over a local ring $(Q,\m,\k)$ in \cite{cid}, proving that this degree is finite whenever $M$ has finite $\text{CI-}$dimension. This paper extends the notion of critical degree via a complete resolution of $M$ over a complete intersection ring of the form $R=Q/(f_1,\dots,f_c)$, thus defining the critical and cocritical degrees of an object $\comp{C}$ in the category of totally acyclic complexes, $\Ktac{R}$. In particular, providing the appropriate dual analogue to critical degree enables us to introduce a new measure for $R$-complexes, called the \emph{critical diameter}.
	\end{abstract}
\maketitle

\section{Introduction}

Interest in the growth of Betti numbers of finitely generated modules over particular classes of commutative rings has been widespread and longstanding. Free resolutions over regular, local rings are finite ($\!\!$\cite{buch-aus}), yet when we take such a ring and quotient out by a regular sequence we often find infinite free resolutions. Let $(Q,\m,\k)$ be a commutative noetherian, local ring with $\dim Q=\dim_{\k}\m/\m^2$ 
and take a complete intersection of the form $R=Q/(f_1,\dots,f_c)$ where $f_1,\dots, f_c$ is a $Q$-regular sequence in $\m$. For this paper, our focus is on growth in the Betti sequence of a finitely generated $R$-module $M$.

Minimal free resolutions of $M$ exhibit realizable patterns in their syzygy sequence. In 1954, Tate showed that $b_n^R(\k)$ is eventually given by a polynomial ($\!\!$\cite{tate}) and, subsequently in 1974, Gulliksen proved that each $b_n^R(M)$ is a quasi-polynomial of period 2 with degree smaller than the codimension ($\!\!$\cite{gu}). Building off of Gullisken's work, Eisenbud demonstrated in 1980 that whenever $R$ is a hypersurface, the free resolution of $M$ is periodic of period $2$ ($\!\!$\cite{Eisenbud}). Later, Avramov demonstrated that $b_n^R(\k)$ has exponential growth \emph{unless} $R$ is a complete intersection ($\!\!$\cite[1.8]{Avramov}). Finally, in 1997, Avramov, Gasharov, and Peeva
demonstrated that, although the beginning of a free resolution over a complete intersection is often unstable, patterns do emerge at infinity. In particular, they  proved that $\{b_n^R(M)\}$ is eventually either strictly increasing or constant. Of special significance is their generalization of modules over a complete intersection to modules of finite \textit{CI-dimension}, for which the same statement holds (see \cite{cid}).

Both complexity and the Hilbert-Poincar\'{e} series help us better measure the asymptotic stability of infinite free resolutions, but an interesting perspective taken by the authors of \cite[\S 7]{cid} was to define the notion \emph{critical degree} of an $R$-module, which essentially represents a flag for when asymptotically stable patterns are guaranteed to arise.  While the periodicity of a free resolution over a hypersurface is guaranteed to start right away ($\!\!$\cite{Eisenbud}), this does not necessarily occur when $\codim(R,Q)\geq2$.

Stability in these resolutions is intimately connected with maximal Cohen-Macaulay (MCM) modules and Tate cohomology, as demonstrated by Buchweitz in 1986 within his article \cite{mcm-tate}. We ask the question: when the syzygy sequence stabilizes to MCM modules, does this necessarily imply that the asymptotic patterns begin? Eisenbud demonstrated that for a hypersurface, the pattern of periodicity arises immediately. 
However, for $\cx_RM\geq 2$, this does not necessarily happen; for example, one could consider a negative syzygy module in the complete resolution of an MCM module. Thus, there is a distinction between the ``head'' of a free resolution carved out by the complete resolution defined by Buchweitz and \emph{when} asymptotic patterns are guaranteed to arise in the syzygy sequence.

It is for this reason that this paper takes the view that there is useful data one could recover about $M$ in the case where $\cx_RM\geq 2$ if we consider \emph{critical degree} with regard to the complete picture of the syzygy sequence.  Thus, our goal is to first extend this notion to the triangulated category of totally acyclic $R$-complexes (which is equivalent to the stable category of MCM modules) and, in doing so, it becomes necessary to present a dual analogue to the notion. In Section \ref{sec:Prelim}, we begin with collecting some preliminary data needed to proceed, and then we complete our primary goal of extending \cite[\emph{Definition 7.1}]{cid} in Section \ref{sec:crcodeg}. 

Afterward, we look towards understanding the natural connection of these definitions to Tate cohomology in Section \ref{sec:cohomchar} and present an appropriate analogue to \cite[\emph{Proposition 7.2}]{cid}. Finally, one pitfall of \emph{critical degree}, as discussed in \cite[\S 7]{cid}, is that for $n>0$, $\Syz{-n}{M}$ will always have a greater critical degree than $M$ despite these modules having the same complexity. It is our hope that by shifting perspective to the \emph{complete} syzygy sequence, we might be able to recover some boundedness properties of (necessarily indecomposable) modules over a fixed ring and given complexity. We present a new invariant of totally acyclic $R$-complexes (and $R$-modules) called \emph{critical diameter} and provide a simple example in Section \ref{sec:diam}.

\section{Preliminaries}\label{sec:Prelim}
Let $(Q,\m,\k)$ be a regular, local ring and $R$ a complete intersection of the form $Q/\f$ where $\f=(f_1,\dots,f_c)$ is a regular $Q$-sequence in $\m$. Recall that an $R$-complex 
$\comp{C}$ is \emph{totally acyclic} if $\H(\comp{C})=0=\H(\comp{C}^*)$ where $\comp{C}^*=\Hom_R(\comp{C},R)$ and each $C_i$ is a finitely generated free\footnote{The definition requires each $C_i$ to be projective, but in our case they will additionally be free.} $R$-module. 
We say that a complex $(\comp{C}, \doug)$ is \emph{minimal} if every homotopy equivalence $\map{e}\!:\comp{C}\to\comp{C}$ is an isomorphism ($\!\!$\cite{rel}), noting this is equivalent to the requirement that $\doug(\comp{C})\subseteq\m\comp{C}$. On the other hand, $\comp{C}$ is \emph{contractible}, or (homotopically) \emph{trivial}, if the identity morphism $\map{1}^{\comp{C}}$ is null-homotopic 
and thus homotopy equivalent to the zero complex. There exists a decomposition of all complexes, $\comp{C}=\mincomp{C} \oplus \comp{T}$, where $\mincomp{C}$ is a unique (up to isomorphism\footnote{If two minimal complexes are homotopy equivalent, then they are isomorphic (c.f. Proposition 1.7.2 in \cite{rel}).}) minimal subcomplex of $\comp{C}$ and $\comp{T}$ is contractible. It further holds that if two complexes are homotopy equivalent, then their minimal subcomplexes must be isomorphic. Throughout this paper, we denote by $\mincomp{C}$ the minimal subcomplex of $\comp{C}$ and it is understood that $\comp{C}\simeq\mincomp{C}$.\footnote{	It is straightforward to check that the natural projection $\pi:\!\comp{C}\to\mincomp{C}$ is a homotopy equivalence.} Note further that a homotopy equivalence between $\comp{C}$ and $\comp{D}$ induces a homotopy equivalence between  $\dual{C}$ and $\dual{D}$. Lastly, a chain map between two complexes always induces a map between any two respective homotopy equivalent complexes, thereby inducing a map on the minimal subcomplexes, as described by the lemma and proposition below.
\begin{lemma}\label{Lem:induced-map}
	Let $\comp{C}$, $\comp{D}$, $\comp{C'}$ and $\comp{D'}$ be $Q$-complexes for which there exist chain maps $\map{f}\!:\comp{C}\to\comp{D}$ and $\gamma\!:\comp{D}\to\comp{D'}$ and further suppose $\comp{C}\simeq\comp{C'}$. Then there exists an induced chain map $\map{f'}\!:\comp{C}' \to\comp{D}'$ such that 
	the square
	\begin{center}
		$\xymatrix{ \comp{C} \ar[d]_{\map{f}}   \ar[r]^{\varphi} & \comp{C}' \ar@{.>}[d]_{\map{f'}} \\ \comp{D}  \ar[r]^{\gamma} & \comp{D}' }$ 
	\end{center}
	commutes (up to homotopy) and this map is unique (up to homotopy).
\end{lemma}
\begin{proof}
	It is straightforward to check that the choice of $\map{f'}=\gamma\map{f}\varphi^{-1}$ makes the square commute up to homotopy and that uniqueness of this map follows.
\end{proof}
\begin{prop}\label{Prop:min-induced-map}
	Let $\comp{C}$ and $\comp{D}$ be $R$-complexes with chain map $\map{f}\!:\comp{C}\to\comp{D}$. If $\mincomp{C}$ and $\mincomp{D}$ are the respective minimal subcomplexes, then there exists an induced chain map $\minmap{f}\!:\mincomp{C}\to\mincomp{D}$, which is unique (up to homotopy).
\end{prop}

\subsection{The Category of Totally Acyclic Complexes} We denote by $\Ktac{R}$, the category of totally acyclic complexes, where the objects are totally acyclic $R$-complexes and the morphisms are homotopy equivalence classes of chain maps. This is a full subcategory of the homotopy category, $\CK(R)$, and so, $\Ktac{R}$ is triangulated with a suspension endofunctor $\shift{}:\! \Ktac{R}\to\Ktac{R}$ taking a complex $(\comp{C},\doug)$ to $(\shift{\comp{C}},\shift{\doug})$ where $(\shift{\comp{C}})_n:= C_{n-1}$ and $(\doug^{\shift{\comp{C}}})_n:=-\doug_{n-1}$. This functor maps the morphism $[f]:\![\comp{C}]\to[\comp{D}]$ to the morphism $\shift{[f]}:\!\shift{\comp{C}}\to\shift{\comp{D}}$, where $\shift{[f]}_n=[f]_{n-1}$ for each $n\in\Z$. Throughout the paper, we discuss ``endomorphisms'' on $\comp{C}$ of degree $-q$ with $q\in\Z^{+}$, which are given by $q$ applications of the suspension endofunctor, denoted by $\shiftq{q}{(-)}$ and it is straightforward to check that if two complexes are homotopy equivalent, say $\comp{C}\homeq\comp{D}$, then $\shiftq{q}{\comp{C}}\homeq\shiftq{q}{\comp{D}}$ for any $q\in\Z^+$.

In Section \ref{sec:crcodeg} of this paper, we additionally need the following lemma, which demonstrates a connection between endomorphisms on two isomorphic complexes in the category of $R$-complexes, $\CR$, which has as morphisms $R$-complex chain maps (\emph{not} homotopy classes).
\begin{lemma}\label{Lem:isom-comp}
	Let $\comp{C}$ and $\comp{D}$ be isomorphic as $R$-complexes in $\CR$. Then there is a one-to-one correspondence between chain endomorphisms on $\comp{C}$ and those on $\comp{D}$. Moreover, an endomorphism on $\comp{C}$ is degree-wise surjective (split injective) if and only if the corresponding endomorphism on $\comp{D}$ is surjective (split injective) at the same degrees.
\end{lemma}
\begin{proof} 
	Take $f\!:\comp{C}\to\comp{D}$ to be a chain map with $f_n: C_n\to D_n $ an $R$-module isomorphism for each $n\in\Z$. Furthermore, let $\mu:\comp{C}\to\shiftq{q}{\comp{C}}$ and $\nu:\comp{D}\to\shiftq{q}{\comp{D}}$, noting that the diagram 
	\begin{center}
		$\xymatrix{ \comp{C} \ar[d]_{\mu} \ar[r]^{f} & \comp{D} \ar@{->}[d]^{\nu} \\ \shiftq{q}{\comp{C}}  \ar[r]^{\shiftq{q}{f}} & \shiftq{q}{\comp{D}} }$ 
	\end{center}
	must commute since $\nu f = ((\shiftq{q}{f})\mu f^{-1})f=(\shiftq{q}{f})\mu$. Thus, it must hold that $\Hom_{\CR}(\comp{C},\shiftq{q}{\comp{C}})\cong\Hom_{\CR}(\comp{D},\shiftq{q}{\comp{D}})$. We show the second part of the lemma with respect to split injectivity, noting that the proof for surjectivity is analogous. First suppose $\nu_n$ is split injective, thus implying the composition $\nu_n f_n=(\shiftq{q}{f})_n \mu_n$ is too. Therefore, $\mu_n$ is split injective. Conversely, if we first assume $\mu_n$ is split injective, then so is the composition $(\shiftq{q}{f})_n \mu_n$ and so there exists a left inverse $\gamma_n:(\shiftq{q}{\comp{D}})_n\to C_n$ such that $\gamma_n\circ(\shiftq{q}{f})_n\mu_n=\Id^{\comp{C}}_n$. This then implies
	\begin{gather*}
	\gamma_n\circ(\nu_nf_n)=\Id^{\comp{C}}_n \\
	f_n\circ\gamma_n\circ(\nu_nf_n)=f_n\circ\Id^{\comp{C}}_n \\
	f_n\circ\gamma_n\circ(\nu_nf_n)=\Id^{\comp{D}}_n\circ f_n \\
	(f_n\circ\gamma_n)\circ\nu_n=\Id^{\comp{D}}_n
	\end{gather*}
	since $f_n$ is right-cancellative. Hence, $f_n\circ\gamma_n$ is a left inverse for $\nu_n$, implying it is split injective, as desired.
\end{proof}

We will forgo the notation signifying equivalence class and it will be assumed that when we refer to $\comp{C}$ or $\map{f}$ in $\Ktac{R}$, we mean the equivalence class or an appropriate representative of the equivalence class. The class of distinguished triangles in $\Ktac{-}$ are determined by triangles of the form 
\begin{gather}\label{eqn:triangle}
\comp{C} \xrightarrow{\map{f}} \comp{D} \xrightarrow{\map{\iota}} \Cone{f} \xrightarrow{\map{\pi}} \shift{\comp{C}}
\end{gather} 
where $\Cone{f}$ is the mapping cone of the chain map $f:\!\comp{C}\to\comp{D}$. 
The reader may refer to Chapter $1$ of \cite{tricat} or \cite{neeTri} for the axioms of triangulated categories. Our interest in the triangulated structure of $\Ktac{R}$ is centered upon the fact that the functors $\HomC{C}{-}\!:\Ktac{R}\to \ab$ and $\HomC{-}{C}\!:\Ktac{R}^{op}\to \ab$ are both cohomological and, furthermore, for any abelian category \Ab, one has $\Ext_{\text{\Ab}}^i(A,B)=\Hom_{\der{\text{\Ab}}}(A,\shiftq{i}{B})$. 

\subsection{Complete Resolutions, Tate Cohomology, and CI Operators} There is a well defined family $\mathfrak{t}=\{\map{t_j}\}$ (with $j=1,\dots,c$) of degree $-2$ chain endomorphisms on any $R$-complex called the \emph{CI operators}\footnote{Also called the \emph{Eisenbud operators.}} originally made explicit in \cite{Eisenbud} for a free resolution. In \cite{gu}, Gulliksen showed that $\Ext_{R}^*(M,N)$ and $\Tor_R^*(M,N)$ are graded modules over a polynomial ring with indeterminants of cohomological degree 2. Following Eisenbud's work, Avramov and Buchwietz provided a new variant of the CI operators, called \emph{cohomology operators} ($\!\!$ \cite[\S 3]{homalgCI-codim2}). Denote by  $\CS=R[\chi_1,\dots,\chi_c]$ the \emph{ring of cohomology operators} and for $R$-modules $M$, $N$ let $$\chi_j:\Ext_R^i(M,N)\to\Ext_R^{i+2}(M,N)$$ 
where $i\in\Z$ and $j=1,\dots,c$ (cf. \cite[\S 4.2]{minfreeresCI}). Furthermore, we may consider the graded module where $N=\k$ and denote by $\Ck=\k[\chi_1,\dots,\chi_c]$. In this case, $\Ext_{R}^*(M,N)$ is additionally a module over $\Ck$ since $\Ext_R^*(M,\k)$ will be annihilated by the maximal ideal $\m$ of $R$.

A \textit{complete resolution} of $M$ is a diagram $\comp{C} \xrightarrow{\rho} \res{P} \xrightarrow{\pi} M$  where $\comp{C}$ is in $\Ktac{R}$, $\res{P}$ is a projective (free) resolution of $M$, and $\rho$ is a morphism of complexes such that $\rho_n$ is bijective for all $n\gg0$ \cite{rel}. 
Any object in $\Ktac{R}$ can be realized as a complete resolution of an $R$-module, $M=\Im\doug^{\comp{C}}_n$, and, conversely, there exists a (unique) minimal free resolution $\res{F}$ for every finitely-generated $R$-module $M$ so we may take this resolution as the projective resolution of $M$ in the diagram. We may then construct a \emph{minimal} totally acyclic complex $\comp{C}$ such that the bijectivity condition between $\comp{C}$ and $\res{F}$ holds ($\!\!$\cite{rel}), yielding a one-to-one correspondence between objects in $\Ktac{R}$ and objects in \Rmod$\!$. It is important to note that within this construction, $\Im\doug^{\comp{C}}_0=\Syz{n}{M}$ such that $\depth_R\Syz{n-1}{M}=\dim R-1$. 


Finally, recall that the \emph{Betti sequence} of $M$ is $\betti{M}$,  where each $b_n:=\rk {C_n}$ or, equivalently, \[b_n=\dim_{\k}\Tateext_R^n(M,\k)\] for each $n\in\Z$. Our interest is in the growth of this sequence, and thus the generators of the syzygy sequence $\{\Omega^nM\}_{n\in\Z}$ (syzygies \emph{and} cosyzygies). It should further be noted that the cohomology operators defined in  \cite{homalgCI-codim2} act on $\comp{C}$ just as they do on $\res{F}$ as degree $2$ chain endomorphisms. We will consider these operators in the case where $N=\k$ so that $\chi_j:=\Hom_R(t_j, \k)$ for $j=1,\dots,c$ with $\codim(R,Q)=c$. It is indeed true that $\Tateext_R^*(M,\k)$ is unambiguously a module over the ring of cohomology operators $\CS=R[\chi_1,\dots,\chi_c]$ as well as $\Ck=\k[\chi_1,\dots,\chi_c]$. In fact, $E_{+}=\bigoplus_{i}\Tateext_R^{i}(M,\k)$ for $i\geq 0$ will be a noetherian module over $\Ck$ while $E_{-}=\bigoplus_{i}\Tateext_R^{i}(M,\k)$ for $i<0$ will be an artinian module over $\Ck$.

\subsection{The Critical Degree of a Finitely Generated Module} For ease of reference, we recall below the definition of critical degree for a finitely generated $Q$-module, as it was originally introduced in \cite[$\S 7$]{cid}. 

\begin{defn}\label{Def:crdeg-modules}
	A $Q$-module $M$ has critical degree of at most $s$, denoted by $\crdeg{}{Q}{M}\leq s$, if its minimal resolution $\mathbb{F}$ has a chain endomorphism $\mu$ of degree $q<0$ such that $\mu_{n-q}: F_{n-q} \to F_n$ is surjective for all $n>s$. If no such $s$ exists, then $\crdeg{}{Q}{M}=\infty$.
\end{defn}

The authors prove in \cite{cid} that for any module of finite CI-dimension, the Betti sequence is non-decreasing past the critical degree, in addition to providing a cohomological characterization for the critical degree. However, they also give an example demonstrating that the critical degree cannot be bounded for all modules of complexity greater than $1$, since taking a cosyzygy out to the right would produce a higher critical degree. Moreover, in \cite[\emph{Example} 7.7]{cid} they show that strict growth does not necessarily signify the critical degree, thus meaning that the critical degree is where growth is \emph{guaranteed} to occur but not necessarily where the growth begins. In \cite[\emph{Theorem} 7.6]{homalgCI-codim2}, the authors give an effective bound on the critical degree of a finitely generated $R$-module of complexity $2$ dependent upon the Betti numbers and $g=\depth R-\depth_RM$.

\subsection{Cosocle, Coregular Sequences, and Codepth of a Module} 
First note that for any $\CS$-module $E$, the \emph{socle} of $E$ is $\soc(E)=\{u\in E\,|\,u\x=0\}$ where $\x=(\chi_1, \dots, \chi_c)$. Dually, (c.f. \cite{RcatMods}), the \emph{cosocle} of $E$
is defined as $$\cosoc(E)=\left\{ \bar{u}\in E/\x E \:|\: \bar{u}\neq 0  \right\}.$$

By \cite[Theorem 2]{DCC-mat}, if $E$ is artinian, then $\x E=E$ if and only if there exists some $x\in \x$ such that $xE=E$. This in turn implies that whenever $\cosoc(E)=0$, there exists some $x\in\x$ such that the submodule generated by $x$ returns the module $E$. Equivalently, this means there is a \emph{coregular} element in $\CS$ ($\!\!$\cite{DCC-mat}, cf. \cite{coreg-quasi}). Recall that a \emph{coregular sequence} in $E$, or an \emph{$E$-cosequence}, is a sequence $\mathbf{\tilde{x}}=x_1,\dots,x_d$ such that
\begin{enumerate}
	\item[(1)] $x_1 E=E$, and
	\item[(2)] $x_i (0:_E (x_1,\dots,x_{i-1}))= (0:_E (x_1,\dots,x_{i-1}))$ for each $i=2,\dots,d$.
\end{enumerate}
Lastly, the \emph{codepth} of $E$, denoted $\cdepth_{\CS}E$, is defined to be the maximal length of an $E$-cosequence in $\x$ ($\!\!$\cite{DCC-mat}, cf. \cite{coreg-quasi} and \cite{Oomat-dual}). It should be clear that $\cdepth_{\CS}E=0$ implies $\cosoc(E)\neq0$, so that existence of some nonzero element $x\not\in\x E$ is guaranteed. We will use this fact in Section \ref{sec:cohomchar} to accomplish our goal of providing a dual analogue for the cohomological characterization of critical degree.

\section{The Critical Degree and Duality}\label{sec:crcodeg}

Our goal in this section is to extend the notion of \emph{critical degree} from free to complete resolutions. We begin with a complex $\comp{C}$ in $\Ktac{R}$ and the natural choice for such an extension.

\subsection{Main Definitions} Given an $R$-complex with $\comp{C}=\Csum{\mincomp{C}}{T}$, then if $\mu\!:\comp{C}\to\shiftq{q}{\comp{C}}$ is a morphism in $\Ktac{R}$, we write $\minmap{\mu}\!:\mincomp{C}\to\shiftq{q}{\mincomp{C}}$ for the induced endomorphism on $\mincomp{C}$, guaranteed by Lemma \ref{Lem:induced-map}.
\begin{defn}\label{Def:crdeg}
	An $R$-complex $\comp{C}$ in $\Ktac{R}$ has \emph{critical degree relative to} $\mu$ (or $\mu$\emph{-critical degree}) equal to \[ \crdeg{\mu}{R}{C} = \mathrm{inf}\{ n \mid \minmap{\mu}_{i+q}: \minmod{C}_{i+q} \twoheadrightarrow \minmod{C}_{i} \: \forall \: i > n \} \]
	and the \textbf{critical degree} of $\comp{C}$ is defined as the infimum over all $\mu$-critical degrees, \[\crdeg{}{R}{\comp{C}}= \mathrm{inf}\{ \crdeg{\mu}{R}{\comp{C}} \: | \: \mu : \comp{C} \to \shiftq{q_\mu}{\comp{C}} \}, \] where $\crdeg{}{R}{\comp{C}}=\infty$ if all possible $\mu$-critical degrees are infinite.
\end{defn}
\begin{rmk}
	Note that for any $R$-complex $\comp{C}$, if we consider the $R$-module $M=\Im\doug^{\comp{C}}_0$, then the above definition of critical degree for $\comp{C}$ will indeed agree with Definition \ref{Def:crdeg-modules} whenever $\crdeg{}{R}{M}\geq0$. If $\crdeg{}{R}{M}<0$, this is not necessarily the case and we later discuss a special case in which $\crdeg{}{R}{M}=-1$ but $-\infty\leq\crdeg{}{R}{C}<-1$.
\end{rmk}

It should further be noted that, as $R$ is a complete intersection, $\crdeg{}{R}{\comp{C}}<\infty$ for two reasons. First, by \cite[Proposition 7.2]{cid} and the remark above. This of course implies there exists an endomorphism $\mu$ realizing this degree, which brings us to the second reason: this $\mu$ must be a linear form of the CI operators, as presented in \cite[\emph{Theorem 3.1}]{Eisenbud}. 

We now define the dual analogue of critical degree and work towards making these definitions precise in $\Ktac{R}$. 
\begin{defn}\label{Def:cocrdeg}
	An $R$-complex $\comp{C}$ in $\Ktac{R}$ has \textit{cocritical degree relative to} $\mu$ (or $\mu$\emph{-cocritical degree}) equal to  \[ \cocrdeg{\mu}{R}{\comp{C}} = \mathrm{sup}\{ n \mid \minmap{\mu}_{i}: \minmod{C}_{i} \hookrightarrow \minmod{C}_{i-q} \text{ splits } \: \forall \: i < n \} \]
	and the \textbf{cocritical degree} of $\comp{C}$ is defined to be the supremum over all $\mu$-cocritical degrees, \[ \cocrdeg{}{R}{\comp{C}}= \mathrm{sup}\{ \cocrdeg{\mu}{R}{\comp{C}} \: | \: \mu : \comp{C} \to \shiftq{q_\mu}{C} \}, \] where $\crdeg{}{R}{C}=-\infty$ if all such relative cocritical degrees are negatively infinite.
\end{defn}

\begin{prop}\label{Prop:hom-stable}
	Let $\comp{C}$ and $\comp{D}$ be  homotopically equivalent $R$-complexes. Then
	$\crdeg{}{R}{\comp{C}}=\crdeg{}{R}{\comp{D}}$ and $\cocrdeg{}{R}{\comp{C}}=\cocrdeg{}{R}{\comp{D}}$.
\end{prop}
\begin{proof}
	The proof of this statement is a straightforward application of Lemmas \ref{Lem:induced-map} and \ref{Lem:isom-comp}
\end{proof}
\begin{cor}\label{Cor:ZeroCrCo}
	If $\comp{C}\simeq 0$, then $\crdeg{}{R}{\comp{C}}=-\infty$ and $\cocrdeg{}{R}{\comp{C}}=\infty$.
\end{cor}
\begin{prop}\label{Prop: periodic}
	Let $\comp{C}$ be in $\Ktac{R}$ with a periodic minimal subcomplex $\mincomp{C}$. Then $\crdeg{}{R}{\comp{C}}=-\infty$ and $\cocrdeg{}{R}{\comp{C}}=\infty$.
\end{prop}
\begin{proof}
	Without loss of generality, assume $\comp{C}$ is minimal and periodic with $\im\doug^{\comp{C}}_n\cong\im\doug^{\comp{C}}_{n+q}$. Then define a degree $q$ endomorphism 
	where $\rho_n:=\id^{C}_n\!:C_n \to C_{n-q}\cong C_n$ for each $n\in\Z$. Hence, $\rho_n$ is both surjective and split injective for all $n\in\Z$, yielding $\crdeg{\rho}{Q}{C}=-\infty$ and $\cocrdeg{\rho}{Q}{C}=\infty$. This in turn forces the critical and cocritical degrees of $\comp{C}$ to be $-\infty$ and $\infty$, respectively.
\end{proof}
As a direct result of Proposition \ref{Prop: periodic}, note that if $R$ is a hypersurface, then $\crdeg{}{R}{\comp{C}}=-\infty$ and $\cocrdeg{}{R}{\comp{C}}=\infty$ for any $R$-complex. This agrees with the bound provided in \cite[Theorem 7.3, Part (1)]{cid}. 

\subsection{Duality} Denote by $M^*=\Hom_R(M,R)$ and note that it is quite easy to see the natural connection between this module and cocritical degree. 
\begin{prop}\label{Prop: Rdual and Cocrdeg}
	Suppose $\crdeg{}{R}{M^*}=s^*$, where $0\leq s^* < \infty$. Then $\cocrdeg{}{R}{\comp{C}}=-s^*-1$ with $\comp{C}$ the complete resolution of $M$. 
\end{prop}
\begin{proof}
	This is a direct application of complete resolutions and the contravariance of the functor $\Hom_R(-,R)$. The negative one-degree shift is due to the relabeling of degrees in $\Hom_R(\res{F}^*,R)$ in the process of taking the complete resolution.
\end{proof}
Now, for any complex $\comp{C}$ in $\Ktac{R}$, denote by $\dual{C}=\Hom_R(\comp{C},R)$ which is also a totally acyclic $R$-complex. Explicitly, if $\comp{C}=(C_n, \doug^{\comp{C}}_n)$, then $\dual{C}=(C^*_n, \doug^{\comp{C}^*}_n)$ is the $R$-complex with
\begin{gather*}
C^*_n=\Hom_R(C_{-n}, R)\cong C_{-n}, \text{ and} \\ 
\doug^{\comp{C}^*}_n=\Hom_R(\doug^{\comp{C}}_{1-n}, R)\cong (\doug^{\comp{C}}_{1-n})^T.
\end{gather*}  where $(-)^T$ represents the transpose. 
It should then be clear that any chain map on $\comp{C}$ does indeed induce a well defined chain map on $\dual{C}$. 
\begin{lemma}\label{Lem:sur-inj-maps1}
	Let $f\!: R^n\to R^m$ be a map between free $R$-modules. Then $f$ is (split) surjective if and only if $f^T$ is split injective. 
\end{lemma}
\begin{proof}	
	This statement follows from the fact that $\Hom_R(-,R)$ is exact on \Rmod and we may apply it to preserve exactness of the appropriate short exact sequences for each direction. The backwards direction follows as $(f^T)^T=f$.
\end{proof}
\begin{cor}\label{Cor:surj-inj-maps2}
	Let $g\!: R^n\to R^m$ be a map between free $R$-modules. Then $g$ is split injective if and only if $g^T$ is (split) surjective.
\end{cor}
\begin{proof}
	Apply previous lemma with $g=(f^T)$ and $g^T=(f^T)^T=f$.
\end{proof}

\begin{prop}\label{Prop:dual-crco}
	Let $\mu\!:\comp{C}\to \shiftq{q}{\comp{C}}$ be a chain endomorphism realizing $\crdeg{}{R}{\comp{C}}=s$ and $\nu\!:\comp{C}\to \shiftq{r}{\comp{C}}$ a chain endomorphism realizing $\cocrdeg{}{R}{\comp{C}}=t$. The \crco of $\dual{C}$ are completely determined by that of $\comp{C}$, where $$\crdeg{}{R}{\dual{C}}=-\cocrdeg{}{R}{\comp{C}} \text{ and}$$ $$\cocrdeg{}{R}{\dual{C}}=q-\crdeg{}{R}{\comp{C}}.$$ 
\end{prop}
\begin{proof}
	First, it should be clear that there is a one-to-one correspondence between endomorphisms on $\comp{C}$ and endomorphisms on $\dual{C}$, as $\Hom_R(\Hom_R(\comp{C},R),R)\cong \comp{C}$ for any $R$-complex $\comp{C}$ in $\Ktac{R}$. Hence, the induced endomorphisms $\mu^*$ and $\nu^*$ must realize the \crco on $\dual{C}$. Apply Corollary \ref{Cor:surj-inj-maps2} to see that $\nu^*_{n+r}=(\nu_{-n})^T$ is surjective whenever $\nu_{-n}$ is split injective, which occurs for all $-n<t$ and thus for all $n>-t$. Then apply Lemma \ref{Lem:sur-inj-maps1} to see that $\mu^*_n=(\mu_{-n+q})^T$ is split injective whenever $\mu_{-n+q}$ is surjective, which occurs for all $-n+q>s$ and thus for all $n<q-s$.
\end{proof}

We now consider how the \crco of an $R$-complex $\comp{C}$ change under the translation endofunctor, $\shift{}\!\!:\Ktac{R}\to\Ktac{R}$. Recall that $\shiftq{n}{\comp{C}}$ denotes the $R$-complex with $R$-modules $(\shiftq{n}{\comp{C}})_m=C_{m-n}$ and differentials $\doug^{\shiftq{n}{\comp{C}}}_m=(-1)^n\doug^{\comp{C}}_{m-n}$. 
\begin{prop}\label{Prop:crco-shifts}
	If $\crdeg{}{R}{\comp{C}}=s$ and $\cocrdeg{}{R}{\comp{C}}=t$, then $\crdeg{}{R}{\shiftq{n}{\comp{C}}}=s+n$ and $\cocrdeg{}{R}{\shiftq{n}{\comp{C}}}=t+n$.
\end{prop}
\begin{proof}
	We begin with assuming $\comp{C}$ is minimal, since the minimal subcomplex of any complex would coincide with a shift of itself under the translation functor. Note that there is a one-to-one correspondence between endomorphisms on $\comp{C}$ and those on $\shiftq{n}{\comp{C}}$. Now suppose $\map{\mu}\!:\comp{C}\to\shiftq{q}{\comp{C}}$ is the endomorphism which realizes the critical degree on $\comp{C}$. 
	Since $(\shiftq{n}{\mu})_i=\mu_{i-n}\!: C_{i-n}\to C_{i-n-q}$ note that $\mu_{i+q}\!:C_{i+q}\to C_i$ surjective for all $i>s$ implies $(\shiftq{n}{\mu})_{i+q}\!:C_{i+q-n}\to C_{i-n}$ will be surjective for all $i>n+s$. Moreover, $s+n$ will be the \emph{least} degree such that $(\shiftq{n}{\mu}_{i+q})$ is surjective for all $i>s+n$ (since $s$ is the least degree such that $\mu_{i+q}$ is surjective for all $i>s$). Hence, $\crdeg{}{Q}{\shiftq{n}{\comp{C}}}=s+n$. Likewise, if $\map{\nu}\!:\comp{C}\to\shiftq{r}{\comp{C}}$ is the endomorphism which realizes the cocritical degree on $\comp{C}$, then $\shiftq{n}{\nu}\!:\shiftq{n}{\comp{C}}\to\shiftq{n+r}{\comp{C}}$ will be such that it realizes the cocritical degree on $\shiftq{n}{\comp{C}}$. Therefore, $(\shiftq{n}{\nu})_i$ will be split injective for all $i<t+n$ and $\cocrdeg{}{R}{\shiftq{n}{\comp{C}}}=t+n$. 
\end{proof}

\subsection{Critical Degrees, Complexes, and Modules}
Starting with a MCM module $M$, our definition \ref{Def:crdeg} agrees with \ref{Def:crdeg-modules} as long as $\crdeg{}{R}{M}\geq 0$. Specifically, if $\cres{C}{F}{M}$ is the minimal complete resolution of $M$ with $\crdeg{}{R}{M}\geq 0$ and $\crdeg{}{R}{M^*}\geq 0$, then $\crdeg{}{R}{\comp{C}}=\crdeg{}{R}{M}$ and $\cocrdeg{}{R}{\comp{C}}=-\crdeg{}{R}{M^*}-1$. On the other hand, if $M$ is \emph{not} MCM, then
\begin{equation*}
\begin{cases}
\crdeg{}{R}{\comp{C}}=s-g & \text{ if } s\geq 0\\
\cocrdeg{}{R}{\comp{C}}=s^*-g^* & \text{ if } s^*\geq 0
\end{cases}
\end{equation*}
where $\crdeg{}{R}{M}=s$ and $g=\dim R-\depth_RM$ (with $s^*$ and $g^*$ analogously defined for $M^*$). 
Observe that, in these cases, $\crdeg{}{R}{\comp{C}}>0$ if $g<s$, and $\cocrdeg{}{R}{\comp{C}}<0$ if $g^*>s^*$. Alternatively, whenever we consider an $R$-module $M$ with $\crdeg{}{R}{M}=-1$ it could be the case that $\crdeg{}{R}{\comp{C}}\lneq -1$, as demonstrated in the following example. 
\begin{example}
	Let $M$ be the $R$-module with complete resolution $\cres{C}{F}{M}$ and further suppose that $0\leq\crdeg{}{R}{M}=s<\infty$. For simplicity, assume also that $M$ is MCM. Now set $N=\Omega^{s+k}M$ for some fixed integer $k>1$, so that $\crdeg{}{R}{N}=-1$. This is because $\mu(F_{n+q})=F_n$ for some $\mu\!:\res{F}\to\shiftq{q}{\res{F}}$ and for all $n>s$ by assumption, but $\res{G}:=\res{F}^{>s+k}$ is the minimal free resolution of $N$. Thus, $G_n=F_{n+s+k}$ so it should be clear that $\mu(G_{n+q})=G_{n}$ for all $n>0$ since $n+s+k>s$. However, note that $N^*=\Hom_R(\Omega^{s+k}M,R)\cong\Omega^{s+k}\Hom_R(M,R)$ and, furthermore, we can complete the chain endomorphism on $\res{F}$ realizing the critical degree of $M$ to a chain map on $\comp{C}$ ($\!$\cite[1.5]{taa}). Of course, there also exists a complete resolution of $N$ of the form $\cres{C}{G}{N}$, which is equivalent to $\cres{C}{\res{F^{\mathit{>s+k}}}}{ \Omega^{s+k}M}$ (up to isomorphism in the first two components and up to homotopy in the last). Therefore, we see that $\crdeg{}{R}{C}\leq-k$ where $-k\lneq -1$ by assumption. $\diamond$
\end{example}
The previous example demonstrates that Definition \ref{Def:crdeg-modules} does not always agree with Definition \ref{Def:crdeg} as $\res{F}$ is a bounded below $R$-complex and so the former definition does not capture what happens in the cosyzygy sequence of $M$. 

\begin{prop}\label{Prop: cocr>cr}
	Let $\comp{C}$ be a minimal complex in $\Ktac{R}$ and $\mu\in\Hom_{\CK(R)}(\comp{C},\shiftq{q}{\comp{C}})$. If $\crdeg{\mu}{R}{\comp{C}}\leq\cocrdeg{\mu}{R}{\comp{C}}$, then one of the following must be true:
	\begin{enumerate}
		\item $0\leq\cocrdeg{\mu}{R}{\comp{C}}-\crdeg{\mu}{R}{\comp{C}}<2+q$, or
		\item $\comp{C}$ is periodic.
	\end{enumerate}
\end{prop}
\begin{proof}
	Set $\crdeg{\mu}{R}{C}=s_{\mu}$ and $\cocrdeg{\mu}{R}{C}=t_{\mu}$ such that $t_{\mu}\geq s_{\mu}$. To demonstrate (1), we show that $t_{\mu}-s_{\mu}\geq 2+q$ forces (2). First, let us suppose equality so that $t_{\mu}= s_{\mu}+(2+q)$ thus implying $\mu_n$ is split injective for all $n<s_{\mu}+(2+q)$ and $\mu_{n+q}$ is surjective for all $n>s_{\mu}$. In particular, $\mu_{t_{\mu}-1}\!:C_{t_{\mu}-1}\to C_{t_{\mu}-1-q}$ is bijective, as $t_{\mu}-1=s_{\mu}+q+1$ and $t_{\mu}-1-q=s_{\mu}+1$. \\
	
	\begin{center}\small
		$\xymatrix @R=1.25cm @C=1.5cm {
			\cdots \ar[r]^(.35){\doug_{(s_{\mu}+3)+q}} & C_{(s_{\mu}+2)+q} \ar@{->>}[d]^{\mu_{(s_{\mu}+2)+q}} \ar[r]^(.4){\doug_{(s_{\mu}+2)+q}} & C_{t_{\mu}-1}=C_{(s_{\mu}+1)+q}   \ar@{^{(}->>}[d]^{\mu_{t_{\mu}-1}=\mu_{(s_{\mu}+1)+q}}_{\cong} \ar[r]^(.6){\doug_{(s_{\mu}+1)+q}} & C_{s_{\mu}+q}   \ar@{^{(}->}[d]^{\mu_{s_{\mu}+q}} \ar[r]^{\doug_{s_{\mu}+q}} 
			& \cdots \\
			\cdots \ar[r]^(.45){\doug_{s_{\mu}+3}} &  C_{s_{\mu}+2}  \ar[r]^(.35){\doug_{s_{\mu}+2}} & C_{t_{\mu}-1-q}=C_{s_{\mu}+1}  \ar[r]^(.6){\doug_{s_{\mu}+1}} & C_{s_{\mu}}  \ar[r]^(.55){\doug_{s_{\mu}}} 
			&\cdots 
		}$ 		
	\end{center} \vspace{2mm} \normalsize
	First note that, by commutativity of the diagram, 
	$$\mu_{(s_{\mu}+1)+q}(\im\doug_{(s_{\mu}+2)+q})
	\subseteq\im\doug_{s_{\mu}+2}=\ker\doug_{s_{\mu}+1}.$$ Conversely, surjectivity of $\mu_{(s_{\mu}+2)+q}$ tells us that $$\ker\doug_{s_{\mu}+1}=\im\doug_{s_{\mu}+2}\subseteq\mu_{(s_{\mu}+1)+q}(\im\doug_{(s_{\mu}+2)+q})$$ since for any $x\in \im\doug_{s_{\mu}+2}$ there exists a $y\in C_{(s_{\mu}+2)+q}$ such that $$\doug_{s_{\mu}+2}\mu_{(s_{\mu}+2)+q}(y)=x=\mu_{(s_{\mu}+1)+q}\doug_{(s_{\mu}+2)+q}(y).$$Hence, $\im\doug_{(s_{\mu}+2)+q}=\ker\doug_{(s_{\mu}+1)+q}\cong\ker\doug_{s_{\mu}+1}=\im\doug_{(s_{\mu}+2)}$ as $\mu_{(s_{\mu}+1)+q}$ is a $Q$-module isomorphism. 
	Now, set $M=\im\doug_{s_{\mu}+1}$, noting that the truncated complex $\comp{C}^{>s_{\mu}}$ is degree-wise bijective with the minimal resolution $\res{F}$ of $M$. Further note that $\Syz{1}{M}=\im\doug_{(s_{\mu}+2)}\cong\im\doug_{(s_{\mu}+2)+q}=\Syz{1+q}{M}$ and so, by uniqueness of $\res{F}$, $\Syz{1+q+n}{M}\cong\Syz{1+n}{M}$ for any $n\in\N$. Now, apply a similar argument to $M^*=\Hom_R(M,R)$ noting that $\Hom_R(\mu_{(s_{\mu}+1)+q},R)$ will likewise be an $R$-module isomorphism from $\Hom_R(C_{s_{\mu}+1},Q)\cong C_{s_{\mu}+1}$ to  $\Hom_R(C_{(s_{\mu}+1)+q},R)\cong C_{(s_{\mu}+1)+q}$. Then, $\Hom_R(\comp{C}^{>s_{\mu}},R)$ is periodic and thus $\comp{C}^{<s_{\mu}}$ is periodic as well. From here, reconstruction of $\comp{C}$ from $\comp{C}^{<s_{\mu}}$ and $\comp{C}^{>s_{\mu}}$ demonstrates that $\comp{C}$ is periodic of period $q$, forcing $\crdeg{}{R}{\comp{C}}=-\infty$ and $\cocrdeg{}{R}{\comp{C}}=\infty$ by Proposition \ref{Prop: periodic}. Note that this argument involved \emph{at least one} degree-wise isomorphism, so it is straightforward to see that the same arguments holds for the case in which $t_{\mu}-s_{\mu}>2+q$.
\end{proof}

\begin{cor}\label{Cor: periodic crco}
	If $\crdeg{}{R}{\comp{C}}$ and $\cocrdeg{}{R}{\comp{C}}$ are realized by the same degree $q$ endomorphism, then $\crdeg{}{R}{\comp{C}}\leq\cocrdeg{}{R}{\comp{C}}-(2+q)$ if and only if $\crdeg{}{R}{\comp{C}}=-\infty$ and $\cocrdeg{}{R}{\comp{C}}=\infty$.
\end{cor}
As a consequence of these results, two questions arise:
\begin{enumerate}[i.]
	\item \emph{Are the \crco of a totally acyclic $R$-complex always realized by the \textbf{same} endomorphism?}
	\item \emph{If not, then is it possible that a \textbf{non}-periodic minimal complex has non-infinite cocritical degree larger than the critical degree by an unbounded amount?}
\end{enumerate}

\section{Tate Cohomology and Critical Degree}\label{sec:cohomchar}
In this section, we have two primary goals: the first is to give an appropriate analogue for \cite[Theorem 7.2]{cid} and the second is to provide a partial answer to question (i.) listed at the end of the previous section. The importance of generalizing the notion of critical degree to $\Ktac{R}$ comes in the form of its triangulated structure, as well as the equivalency of categories $$\Ktac{R}\simeq\stMCM{R}$$ as each finitely generated $R$-module has a MCM approximation ($\!\!$\cite{mcm-tate}). It is then unsurprising that we may relate critical degree to $$\Tateext_R^n(M,\k)=\H^n\Hom_R(\comp{C},\k)$$ equipped with comparison homomorphisms $$\varepsilon_R^n(M,\k):\Ext^n_R(M,\k)\to\Tateext_R^n(M,\k)$$ as described in \cite[\S 5]{rel} (cf.\cite{mcm-tate}, \cite{gorproj-comp}, \cite{pten}).  \begin{lemma}\label{Lem:Ext-compVmod}
	Let M be maximal Cohen-Macaulay. Given complete resolutions $\compK\xrightarrow{\rho}\res{F}\to \k$ and  $\comp{C}\to\res{G}\to M$, there exists an isomorphism $$\Ext_{\CK(R)}^n(\comp{C},\compK)\cong\Tateext_{R}^n(M,\k)$$
	for each $n\in\Z$.
\end{lemma}
\begin{proof}
	The isomorphism is given by 	$$\xymatrix{\Ext_{\CK(R)}^n(\comp{C},\compK) \ar@{<:>}[r] 
	\ar@{=}[d] & \Tateext_{R}^n(M,\k) \ar@/_/[d]_{\Phi_n}  \\ \Hom_{\CK(R)}(\comp{C},\shiftq{n}{\compK}) \ar@{=}[r] 
	& \Hom_{\CK(R)}(\mincomp{C},\shiftq{n}{\mincomp{K}}) \ar@/_/[u]_{\Theta_n}^{\cong} 
	}$$ where the first (left-hand) equality is by definition (cf. \cite[(1.3)]{tate}) and the second (bottom) equality is due to the fact $\comp{C}\cong\mincomp{C}$ in $\CK(R)$. The third (right-hand) isomorphism is by the Comparison Theorem (cf. \cite[Lemma 5.3]{rel}) where $\Phi_n$ extends any $R$-module map $f:M_n\to\k$ to a chain map $\hat{f}:\mincomp{C}\to\shiftq{n}{\mincomp{K}}$ and $\Theta_n$ restricts in the other direction, as indicated by the diagram 
	$$\xymatrix{
	\comp{C} \ar[r] \ar@{..>}[d]^{\hat{f}} & (\res{F})^{\geq n} \ar@{..>}[d]^{\bar{f}} \ar[r] & M_n \ar[d]^{f\vert_{M_n}} \\ \compK \ar[r] & \res{K} \ar[r] & \k \\	}$$ with $M_n=\coker(\doug^{\comp{C}}_{n+1})$.
\end{proof}

The intermediary step to providing a cohomological characterization of critical degree in $\Ktac{R}$ yields a triangulated approach to Definitions \ref{Def:crdeg} and \ref{Def:cocrdeg}. Take $\compK\to\res{K}\dhrightarrow\k$ to be the \emph{minimal} complete resolution of the residue field $\k=R/\m$ and note that for any endomorphism $\map{\mu}\!:\comp{C}\to\shiftq{q}{C}$, the distinguished triangle $$\tri{\comp{C}}{\mu}{\shiftq{q}{\comp{C}}}{\iota}{\Cone{\mu}}{\pi}{C}$$ yields the long exact sequence of abelian groups
\begin{align*}
	\cdots\to\Ext_{\CK(R)}^n(\Cone{\mu},\compK)\to\Ext_{\CK(R)}^n(\shiftq{q}{\comp{C}},\compK)\xrightarrow{\mu^{n}}\Ext_{\CK(R)}^n(\comp{C},\compK) \\ \to\Ext_{\CK(R)}^{n+1}(\Cone{\mu},\compK)\to\Ext_{\CK(R)}^{n+1}(\shiftq{q}{\comp{C}},\compK)\to\cdots
\end{align*}
\normalsize
where $\mu^{n}=\Hom_{\CK(R)}(\mu_{n},\compK)$ (cf. \cite[\emph{Chapter 1, Proposition 4.2}]{tricat}). 
\begin{prop}\label{Prop:tri-cr&cocr}
	The $\mu$\emph{-critical degree of} $\comp{C}$ is the least homological degree $s_{\mu}$ for which $\mu^{n+q}$ is (split) injective for all $n>s_{\mu}$,
	$$\crdeg{\mu}{R}{\comp{C}}=\inf\{i\:\vline\: \mu^{n+q}\!:\Ext_{\CK(R)}^{n+q}(\shiftq{q}{\comp{C}},\compK)\hookrightarrow\Ext_{\CK(R)}^{n+q}(\comp{C}, \compK) \text{ for all } n>i \}.$$
	Likewise, the $\mu$\emph{-cocritical degree of} $\comp{C}$ is the greatest homological degree $t_{\mu}$ for which $\mu^n$ is (split) surjective for all $n<t_{\mu}$,
	$$\cocrdeg{\mu}{R}{\comp{C}}=\sup\{i\:\vline\: \mu^n\!:\Ext_{\CK(R)}^{n}(\shiftq{q}{\comp{C}}, \compK)\twoheadrightarrow\Ext_{\CK(R)}^{n}(\comp{C}, \compK) \text{ for all } n<i \}.$$
\end{prop}
\begin{rmk}
	If there exists no such infimum $s_{\mu}$, then, by definition, $\crdeg{\mu}{R}{\comp{C}}=\infty$ and, similarly, if there exists no such supremum $t_{\mu}$ then $\crdeg{\mu}{R}{\comp{C}}=-\infty$. Further note that we may define the \textbf{critical degree of} $\mathbf{\comp{C}}$ in the same way as Definition \ref{Def:crdeg}
	and the \textbf{cocritical degree of} $\mathbf{\comp{C}}$ in the same way as Definition \ref{Def:cocrdeg}. The only distinction here is how we are defining the \emph{relative} \crco for a given $R$-complex and chain endomorphism in $\Ktac{R}$.
\end{rmk}

\begin{proof}[Proof of Proposition \ref{Prop:tri-cr&cocr}]
	Without loss of generality, assume $\comp{C}$ is minimal. First note that Lemma \ref{Lem:Ext-compVmod} gives us the induced map $\hat{\mu}^n$ for any $n\in\Z$, as depicted in the diagram
	$$\xymatrix{\Ext_{\CK(R)}^n(\shiftq{q}{\comp{C}},\compK) \ar[r]^{\mu^n} \ar@/^/[d]^{\Phi_j} & \Ext_{\CK(R)}^n(\comp{C},\compK) \ar@/_/[d]_{\Phi_n}  \\ \Tateext_{R}^{n}(M_{-q},\k) \ar@{..>}[r]^{\hat{\mu}^n} \ar@/^/[u]^{\Theta_n}_{\cong} & \Tateext_R^{n}(M,\k) \ar@/_/[u]_{\Theta_n}^{\cong}
	}$$
	where $M_{-q}:=\Im\doug^{\shiftq{q}{\comp{C}}}_0$. Hence, it is enough to show that, for $n\in\Z$, $\mu_n$ is surjective (split injective) if and only if $\hat{\mu}^n$ is split injective (surjective). We begin with the forward direction of both cases. First note that $\mu_n$ surjective yields the short exact sequence 	\[0\to\ker(\mu_{n})\to C_{n}\xrightarrow{\mu_{n}} C_{n-q}\to 0\] 
	which splits. Then, applying $\Hom_R(-,\k)$, we see that the short exact sequence
	\[0\to \Hom_R(C_{n-q},\k)\xrightarrow{\Hom(\mu_{n},\k)}\Hom_R(C_{n},\k)\to \Hom_R(\ker(\mu_{n}),\k)\to 0\]
	must split as well and, thus, $\hat{\mu}^n=\Hom(\mu_{n},\k)$ is split injective. On the other hand, if we suppose $\mu_n$ is split injective then the short exact sequence \[0\to C_n\xhookrightarrow{\mu_n} C_{n-q}\to A \to 0\] splits, where $C_{n-q}=A\oplus \im(\mu_n)$. Once again, applying $\Hom_R(-, \k)$ yields the short exact sequence  \[0\to\Hom_R(A,\k)\to\Hom_R(C_{n-q},\k) \xrightarrow{\hat{\mu}^n}\Hom_R(C_n,\k) \to 0\] which must also split, implying that $\hat{\mu}^n$ is surjective as needed. Now, to address the backwards direction, first observe that $\hat{\mu}^n$ split injective is equivalent to being left-cancellative so that $	\hat{\mu}^{n}(\alpha)=\hat{\mu}^{n}(\beta)$ implies $\alpha=\beta$ for any two cocycles $\alpha$, $\beta\in\Hom_R(C_{n-q}\,\k)$. Of course, due to the action of $\hat{\mu}^n$ on $\Tateext_{R}(M_{-q},\k)$, this is equivalent to the statement that $\alpha\mu_{n}=\beta \mu_{n}$ implies $\alpha=\beta$ for any two morphisms $\alpha$, $\beta\!:C_{n-q}\to\k$. Equivalently, $\mu_{n}$ is right-cancellative, and thus surjective. \\
	
	The same approach does not work for $\hat{\mu}_n$ surjective, so instead denote by $\{e_i\}$ a basis for $C_n$ and define $\pi_j\!\!:C_n\to\k$ so that $\pi_j(e_i)=\delta_{ij}$. Note that since $\doug(\comp{C})\subseteq \m\comp{C}$, each $\pi_j$ will be a cocycle in $\Hom_R(C_n,\k)$. By the surjectivity of $\hat{\mu}^n$ there exist maps $\rho_j\!\!:C_{n-q}\to\k$ such that $\pi_j=\rho_j\mu_n$ in turn implying $\{\mu_n(e_i)\}$ forms a linearly independent subset of a basis $\mathcal{E}$ for $C_{n-q}$ by Nakayama.
	As a linearly independent sub-basis of $C_{n-q}$ (defined by $\im(\mu_n)$) is in one-to-one correspondence with a basis for $C_n$, it must hold that $\mu_n$ is injective. Furthermore,
	$\mathcal{E}=(\mathcal{E}\setminus\{\mu_n(e_i)\})\cup\{\mu_n(e_i)\}$ and, if we denote by $A$ the subspace generated by $\mathcal{E}\setminus\{\mu_n(e_i)\}$, then we may write $C_{n-q}=A\oplus\im(\mu_n)$ showing $\mu_n$ is split injective.	
\end{proof}


\subsection{Cohomological Characterizations} We are now ready to present our results, analogous to Proposition 7.2 in \cite{cid}, which give the cohomological characterizations of the critical and cocritical degrees in $\Ktac{R}$.
\begin{prop}\label{Prop:cr-cohom}
	If $\comp{C}\not\simeq0$ is a totally acyclic $R$-complex with $M=\im\doug^{\comp{C}}_0$, then $\crdeg{}{R}{\comp{C}}<\infty$ and the following equalities hold:
	$$\crdeg{}{R}{\comp{C}}=\sup\{r\in\Z\:|\:\depth_{\CS}\Ext_{\CK(R)}^{\geq r}(\comp{C},\compK)=0\}$$
	$$\hspace{10mm}=\sup\{r\in\Z\:|\:\depth_{\CS}\Tateext_{R}^{\geq r}(M,\k)=0\}.$$
\end{prop} 

In the proposition above, recall that $\gExt{\CK(R)}{\comp{C}}{\compK}$ is a graded $\CS$-module where $\CS=R[\chi_1,\dots,\chi_c]$ and denote by $\x=(\chi_1,\dots,\chi_c)$ (cf. \cite[\S 4.2]{minfreeresCI}). Thus, for each $r\in\Z$, note that $\Ext_{\CK(R)}^{\geq r}(\comp{C},\compK)$ is a submodule containing all elements of degree greater than or equal to $r$. Note that we only need show the second equality, as the first equality will then be given by Proposition \ref{Prop:tri-cr&cocr}. In the following proof, we handle the cases for $\crdeg{}{R}{\comp{C}}\geq 0$ and $\crdeg{}{R}{\comp{C}}< 0$ separately, as the connection to Definition \ref{Def:crdeg-modules} simplifies the former case. For ease of notation, set $\crdeg{}{R}{\comp{C}}=s$ and $$r^*=\sup\{r\in\Z\:|\:\depth_{\hat{\CS}}\Tateext_{R}^{\geq r}(M,\k)=0\}.$$

\begin{proof}
	Finiteness of $s$ is guaranteed by the existence of a linear form $\ell\in\CS$ which is eventually a non zero-divisor on the truncation $\Ext_R^{\geq N}(M,\k)$ for some $N\gg0$ ($\!\!$\cite[Theorem 3.1]{Eisenbud}). Whenever $s\geq 0$, $\crdeg{}{R}{\comp{C}}=\crdeg{}{R}{M}$ and we can apply \cite[(7.2)]{cid} to see that $$s=\sup\{r\in \N\cup\{0\}\,|\, \depth_{\CS}\Ext_R^{\geq r}(M,\k)=0\}.$$ It should be clear that, in this case, $r^*=s$ since $\Tateext_{R}^{\geq r}(M,\k)=\Ext_{R}^{\geq r}(M,\k)$ for each $r\geq 0$. Likewise, if $r^*\geq 0$, this forces $s\geq 0$ and $r^*=s$. So now suppose $r^*\leq s<0$ and consider the $R$-complex $\shiftq{\abs{r^*}}{\comp{C}}$ which has non-negative critical degree since $\crdeg{}{R}{\shiftq{\abs{r^*}}{\comp{C}}}=s\: +\abs{r^*}$ by Proposition \ref{Prop:crco-shifts}. As $\im\doug^{\shiftq{\abs{r^*}}{\comp{C}}}_0=\im\doug^{\comp{C}}_{r^*}$, it should be clear that $$s\:+\abs{r^*}=\sup\{r\in\N\cup\{0\}\,|\,\depth_{\CS}\Ext_{R}^{\geq r}(M_{r^*},\k)=0\}$$
	and so there is a nonzero socle element of degree $s\:+\abs{r^*}$. However, there exists some $\chi\in\x$ that is a non zero-divisor on $\Ext_R^{>s+\abs{r^*}}(M_{r^*},\k)$ meaning $$\depth_{\CS}\Ext_R^{>s+\abs{r^*}}(M_{r^*},\k)\neq0.$$ By \cite[\emph{Lemma 4.3}]{tate}, there is an isomorphism $$\Tateext_R^{s+\abs{r^*}+i}(M_{r^*},\k)\cong\Tateext_R^{s+i}(\Omega^{\abs{r^*}}M_{r^*},\k)$$ thus implying $$\Ext_R^{s+\abs{r^*}+i}(M_{r^*},\k)\cong\Tateext_R^{s+i}(M,\k)$$ for each $i\in\Z$ as well. 
	Therefore, $$\depth_{\CS}\Tateext_R^{\ge s+i}(M,\k)\neq 0$$ for all $i>0$ but it must hold that $\depth_{\CS}\Tateext_R^{\ge s}(M,\k)=0$, in turn implying $r^*=s$. If we instead suppose $s\leq r^*<0$ and consider the $R$-complex $\shiftq{\abs{s}}{C}$, we can apply the same argument as above since $\crdeg{}{R}{\shiftq{\abs{s }}{C}}=0$ and there exists an isomorphism $$\Tateext_R^{s+\abs{s}+i}(M_s,\k)\cong\Tateext_R^{s+i}(\Omega^{\abs{s}}M_s,\k)$$ so that $$\Tateext_R^{i}(M_s,\k)\cong\Tateext_R^{s+i}(M,\k)$$ for each $i>0$.
\end{proof}

We now work to show the cohomological characterization for cocritical degree, which aligns nicely with what we understand about critical degree. Before doing so, we first need a lemma that establishes a significant relationship between $\gTateext{R}{M}{\k}$ and $\gTateext{R}{M^*}{\k}$.

\begin{lemma}\label{Lem:depth-codepth}
	Let $\comp{C}\not\simeq0$ be in $\Ktac{R}$ with $\cres{C}{F}{M}$ the minimal complete resolution of $M=\im\doug^{\comp{C}}_0$ and $\shiftq{-1}{\comp{C}^*}\to(\shiftq{-1}{\comp{C}^*})^{\geq 1}\twoheadrightarrow M^*$ the (minimal) complete resolution of $M^*\cong\im(\shiftq{-1}{\doug^*})_0=\im\doug^*_1$. Then: $$\depth_{\CS}\Tateext_R^{\geq r}(M^*,\k)=0 \text{ if and only if } \cdepth_{\CS}\Tateext_R^{\leq -r}(M,\k)=0.$$
\end{lemma}
\begin{proof}
	First note that $\depth_{\CS}\Tateext^{>r}_R(M^*,\k)\neq0$ if and only if there exists some $\chi\in\x$ which is a non zero-divisor on $\Tateext^{>r}_R(M^*,\k)$. This is equivalent to multiplication by $\chi$ on $\gExt{R}{M^*}{\k}$ being injective for elements of degree $n\geq r$. Denote by $q=\deg(\chi)$ and define $(\hat{\mu}^*)^n\!:\Tateext_{R}^n(M^*_{-q},\k)\to\Tateext_{R}^n(M^*,\k)$ by multiplication of $\chi$. By the same argument in the proof of Proposition \ref{Prop:tri-cr&cocr}, $(\hat{\mu}^*)^n$ is injective if and only if $(\mu^*)_{n+q}\!:C^*_{n+q}\to C^*_n$ is surjective for each $n\geq r$, where $(\hat{\mu}^*)^{n+q}=\Hom(\mu^*_{n+q},\k)$. Furthermore, by Lemma \ref{Lem:sur-inj-maps1}, this is equivalent to the map $\mu_{n}=(\mu^*_{n+q})^T\!:C_{n}\to C_{n-q}$ being split injective for $n\leq -r$. Now, applying the same argument from Proposition \ref{Prop:tri-cr&cocr} once more, we see that this is equivalent to the map $\hat{\mu}^n=\Hom(\mu_n,\k)\!:\Tateext_{R}^n(M_{-q},\k)\to\Tateext_{R}^n(M,\k)$ being surjective for $n\leq -r$. Equivalently, there exists an element $\chi'\in\x$ for which multiplication by $\chi'$ on $\gExt{R}{M}{\k}$ is surjective for all elements of degree $n\leq -r$. Lastly, $\chi' \Tateext^{<-r}_R(M,\k)=\Tateext^{<-r}_R(M,\k)$ if and only if 
	$\cdepth_{\CS}\Tateext^{<-r}_R(M,\k)\neq 0$. 
\end{proof}
We provide the following visual proof of Lemma \ref{Lem:depth-codepth}, as well as its relevancy to the \crco$\!\!$:\vspace{2mm}

\tiny
\begin{center}
	$\xymatrix @R=.75cm @C=2.25cm { *++[F]{\text{Surjective for all } n>r} \ar@/^3pc/@{<~>}[r]|{\text{dualizing}} \ar@{<=>}[r]^{\text{Lemma \ref{Lem:sur-inj-maps1}}} & *++[F]{\text{Split Injective for all } n<-r} \ar@{<:>}[d]|{\text{Prop. \ref{Prop:tri-cr&cocr}}} \\ *++[F.]{\textbf{ \textcolor{gray}{Injective on L.E.S.}}} \ar@{<:>}[u]|{\text{Prop. \ref{Prop:tri-cr&cocr}}} & *++[F.]{\textbf{ \textcolor{gray}{Surjective on L.E.S.}}} \ar@{<=>}[d]|{\text{\color{gray}(Equivalent)}} \\ *++[F=]{\textbf{ \textcolor{gray}{Regular element on }} \Ext_R^{>r}(M^*,\k)} \ar@{<=>}[u]|{\text{\color{gray}(Equivalent)}} \ar@{<.>}
		[r]^{\text{Equivalency}}_{\text{of notions}} & *+++[F=]{\textbf{ \textcolor{gray}{Coregular element on }} \Tateext_R^{<-r}(M,\k)} \\ *++[F.]{\textbf{\textcolor{gray}{Highest Degree Socle Element}}} \ar@{<:}[u]|{\color{gray}(\crdeg{}{R}{M^*})} & *+++[F.]{\textbf{\textcolor{gray}{Lowest Degree Cosocle Element}}} \ar@{<:}[u]|{\color{gray}(\cocrdeg{}{R}{M})} 
	}$
\end{center}\vspace{2mm}\normalsize

\begin{prop}\label{Prop:cocr-cohom}
	If $\comp{C}\not\homeq 0$ is a totally acyclic $R$-complex with $M=\im\doug^{\comp{C}}_0$, then $\cocrdeg{}{R}{C}=t>-\infty$ and the following equalities hold:
	$$\cocrdeg{}{R}{C}=\inf\{r\in\Z\:|\:\cdepth_{\hat{\CS}}\Ext_{\CK(R)}^{\leq r}(\comp{C},\comp{K})=0\}$$
	$$\hspace{10mm}=\inf\{r\in\Z\:|\:\cdepth_{\hat{\CS}}\Tateext_{R}^{\leq r}(M,\k)=0\}.$$
\end{prop}

\begin{proof}
	Note that if we take $s+1$ to be the \emph{lowest} degree such that there exists a non zero-divisor on $\Tateext_R^{>r}(M^*,\k)$, then $s$ is the highest degree of a nonzero element in $\soc(\gTateext{R}{M^*}{\k})$. And since this correlates to the \emph{highest} degree for which there exists a generating (coregular) element, apply \cite[\emph{Theorem 2}]{DCC-mat} to see that $-s$ should be the \emph{lowest} degree for which there exists a nonzero element in $\cosoc(\gTateext{R}{M}{\k})$. Further note that since $\depth_{\ZZ}\Ext_R(M^*,\k)$ coincides with $\depth_{\CS}\Ext_R(M^*,\k)$ by \cite[7.2]{cid}, there must exist a nonzero coregular element from $\x$ on the greatest truncation of $\Tateext_R^{<-r}(M,\k)$ for which there exists such a generating element. That is to say, there exists some $\chi\in\x$ such that $\chi\Tateext_R^{\leq t}(M,\k)=\Tateext_R^{\leq t}(M,\k)$ where $t=\cocrdeg{}{R}{C}$. Thus, the cocritical degree of a complex over a complete intersection is (negatively) finite and, 
	moreover, the equalities in the proposition hold.
\end{proof}

\subsection{An Eventual Degree-wise Epimorphism and Monomorphism} We now explore question (i.) presented at the end of Section \ref{sec:crcodeg}. Our goal is to employ an analogous argument to the proof of \cite[Thereom 3.1]{Eisenbud} in order to demonstrate existence of an endomorphism on $\comp{C}$ in $\Ktac{R}$ that is eventually degree-wise surjective on the left and split injective on the right.
\begin{thm}
	Let $Q$ be a commutative local, regular ring with infinite residue $\k$ and $R$ a complete intersection of the form $Q/\f$ where $\f$ is a regular $Q$-sequence and $\codim(R,Q)=c$. Further let $\comp{C}$ be a totally acyclic $R$-complex with minimal subcomplex $\mincomp{C}$. Then there exists a degree $2$ chain endomorphism such that $$\bfu\!:\bar{C}_{n+2}\to \bar{C}_n$$ is an epimorphism for all $n\gg 0$ and a monomorphism for all $n+2\ll 0$.
\end{thm}
\begin{proof}
	First, set $M=\Im\doug^{\comp{C}}_0$ and $M^*=\Hom_R(M,R)$. Our goal is to demonstrate that there exists some $\chi$ in the ideal $\x$ of $\Ck=\k[\chi_1,\dots,\chi_c]$ such that $\chi$ is both a non zero-divisor on $\Tateext_R^{>s}(M,\k)$ for some $s\in\Z$ and $\chi\Tateext_R^{<t}(M,\k)=\Tateext_R^{<t}(M,\k)$ for some $t\in\Z$. As  $\gTateext{R}{M}{\k}$ and $\gTateext{R}{M^*}{\k}$ are both graded modules over $\Ck$, the action of this ring remains the same on each of the modules. For ease of discussion, denote by $E_1=\Tateext_{R}^{\geq 0}(M,\k)=\Ext_{R}^{\geq 0}(M,\k)$ and $E_2=\Tateext_{R}^{\geq 0}(M^*,\k)=\Ext_{R}^{\geq 0}(M^*,\k)$ so that $E_1$ and $E_2$ are both noetherian modules over $\Ck$. Furthermore, denote by $A_i=\soc(E_i)$ for each $i=1,2$ and note each $A_i$ is an artinian submodule of the respective graded Tate module. Now, as each $A_i$ is both artinian and noetherian, it is of finite length, implying there exists some $$N_i=\sup\{n\in\N\cup\{0\}\,|\,\deg(x)=n \text{ for } x\in A_i\}$$ for each $i=1,2$. Set $E_1^{>N_0}=\Ext_{R}^{>N_0}(M,\k)$ and $E_2^{>N_0}=\Ext_{R}^{>N_0}(M^*,\k)$ where $N_0=\max\{N_1,N_2\}$, so that neither truncation contains a nonzero element annihilated by $\x$. Then denote by $P_1,\dots,P_q$ the associated primes of $0\in E_1^{>N_0}$ and $Q_1,\dots,Q_r$ the associated primes of $0\in E_2^{>N_0}$. Hence, $$P_1\cup\cdots\cup P_q\cup Q_1\cup\cdots\cup Q_r$$ contains the set of zero-divisors on both $E_1^{>N_0}$ and $E_2^{>N_0}$. Now consider the set $$\chi_1+\sum_{i=2}^c\k \chi_i,$$
	which generates $\x$, so it cannot be contained in any $P_k$ or $Q_k$ since there is no element of $\x$ which is a zero-divisor on $E_i^{\geq N_0}$. Moreover, $\k$ is infinite, thus yielding a translation of this set 
	which is a subspace of $\Ck$ and this subspace cannot be contained within $P_1\cup\cdots\cup P_q\cup Q_1\cup\cdots\cup Q_r$. Meaning, there exists a linear form $$\hat{\chi}=\chi_1+\sum_{i=2}^c \alpha_j \chi_j$$ with $\alpha_j\in\k$ such that $\chi$ is a non zero-divisor on both $E_1^{\geq N_0}$ and $E_2^{\geq N_0}$. For each $j=1,\dots,c$ set $a_j$ equal to a pre-image of $\alpha_j$ in $R$ so that $$\chi=\chi_1+\sum_{i=2}^c a_j \chi_j\in\CS.$$
	Lastly, note that $\chi$ is a non zero-divisor on $E_1^{\geq N_0}$ if and only if $$\chi^{n+2}\!:\Ext_R^{n}(M,\k)\to\Ext_R^{n+2}(M,\k)$$ is injective for all $n>N_0$ if and only if $$\bfu=u_1+\sum_{i=2}^c a_j u_j$$ is surjective for all $n>N_0$, where $\chi=\Hom_R(\bfu, \k)$. Similarly, we see that $\chi$ is additionally a non zero-divisor on $E^{\geq N_0}_2=\Ext_R^{\geq N_0}(M^*,\k)$. As demonstrated in Lemma \ref{Lem:depth-codepth}, there is a one-to-one correspondence between regular elements on $\Ext_R^{\geq N_0}(M^*,\k)$ and coregular elements on $\Tateext_R^{\leq -N_0}(M,\k)$. Consequently, $\chi$ is ``generating element'' of $\Tateext_R^{\leq -N_0}(M,\k)$ if and only if $$\chi^{n+2}\!:\Ext_R^{n}(M,\k)\to\Ext_R^{n+2}(M,\k)$$ is surjective for all $n+2<-N_0$ if and only if
	$$\chi^m\!:\Ext_R^{m-2}(M,\k)\to\Ext_R^{m}(M,\k)$$  
	is surjective for all $m<-N_0$ if and only if the linear form $\bfu\!:\bar{C}_{n}\to \bar{C}_{n-2}$ is split injective for all $n<-N_0$.
\end{proof}

While this illustrates that at least one of $\crdeg{}{R}{\comp{C}}$ or $\cocrdeg{}{R}{\comp{C}}$ is realized by either $N_1$ or $-N_2$, the question remains on whether $N_1=N_2$ in the proof above, implying that the critical and cocritical degrees are realized by the \emph{same} degree $2$ endomorphism. We suspect this holds at least for the complete resolution of the residue field.

\begin{cor}\label{cor:CIop-crcodeg}
	Let $Q$ be a commutative local, regular ring with infinite residue $\k$ and $R$ a complete intersection of the form $Q/\f$ where $\f$ is a regular $Q$-sequence and $\codim(R,Q)=c$. Further let $\comp{C}$ be a totally acyclic $R$-complex with minimal subcomplex $\mincomp{C}$. Then one of the following cases must hold:
	\begin{enumerate}
		\item[(i.)] If $\crdeg{\mu}{R}{\comp{C}}=\crdeg{}{R}{\comp{C}}$, then $\mu$ is a linear form in $\Ck$ and $$\cocrdeg{\mu}{R}{\comp{C}}\leq\cocrdeg{}{R}{\comp{C}}.$$
		\item[(ii.)] If $\cocrdeg{\mu}{R}{\comp{C}}=\cocrdeg{}{R}{\comp{C}}$, then $\mu$ is a linear form in $\Ck$ and $$\crdeg{\mu}{R}{\comp{C}}\geq\crdeg{}{R}{\comp{C}}.$$
	\end{enumerate}
\end{cor}

\section{Basic Operations in $\Ktac{R}$ and Critical Diameter}\label{sec:diam}

\subsection{Direct Sums and Retracts} We begin with discussing critical degree of the (co)product in $\Ktac{R}$, given by the direct sum of two $R$-complexes 
$$\xymatrix @R=0.25cm @C=0.25cm { \comp{C} \ar@{_{(}->} [dr]^{\iota_{\comp{C}}} & & \comp{D} \ar@{^{(}->} [dl]_{\iota_{\comp{D}}} \\ &  \Csum{C}{D} \ar@{->>}[dl]^{\pi_{\comp{C}}} \ar@{->>}[dr]_{\pi_{\comp{D}}} & \\ \comp{C}& & \comp{D} }$$
where $\Csum{C}{D}$ is the complex with $R$-modules $(\Csum{C}{D})_n=C_n\oplus D_n$ and $R$-module homomorphisms 
$$\doug^{\Csum{C}{D}}_n=\begin{pmatrix}
\doug_n^{\comp{C}} & 0 \\ 0 & \doug_n^{\comp{D}}
\end{pmatrix}.$$
Note that minimality is preserved over direct sums, where the decomposition of $\Csum{C}{D}$ is given by $\Csum{(\mincomp{\Csum{C}{D}})}{\comp{T}}$ with $\mincomp{\Csum{C}{D}}=\Csum{\mincomp{C}}{\mincomp{D}}$ and $\comp{T}$ contractible.

\begin{prop}\label{Prop:sums-crco}
	Let $R$ be a complete intersection of the form $Q/(\f)$, with $\f=f_1,\dots, f_c$ a regular $Q$-sequence. Further suppose $\comp{C}\in\Ktac{R}$ and $\comp{D}\in\Ktac{R}$, so that $\Csum{C}{D}\in\Ktac{R}$. Denote by $\crdeg{}{R}{\comp{C}}=s_1$, $\crdeg{}{R}{\comp{D}}=s_2$, $\cocrdeg{}{R}{\comp{C}}=t_1$, and $\cocrdeg{}{R}{\comp{D}}=t_2$. Then: $$\crdeg{}{R}{(\Csum{C}{D})}=\max\{s_1,s_2\},\text{ and }$$ $$\cocrdeg{}{R}{(\Csum{C}{D})}=\min\{t_1,t_2\}.$$
\end{prop}
\begin{proof}
	Suppose $s=\crdeg{}{R}{(\Csum{C}{D})}$. Further let $M=\Im\doug^{\comp{C}}_0$ and $N=\Im\doug^{\comp{D}}_0$, so that $M\oplus N=\Im\doug^{\Csum{C}{D}}_0$ noting that there exists an isomorphism between the $\CS$-modules $\Tateext_R^*(M\oplus N, \k)\cong\Tateext_R^*(M, \k)\oplus\Tateext_R^*(N, \k)$. So take $0\neq\mathfrak{x}\in\soc(\gTateext{R}{M}{\k})$ with $\deg(\mathfrak{x})=s_1$ and $0\neq\mathfrak{z}\in\soc(\gTateext{R}{N}{\k})$ with $\deg(\mathfrak{z})=s_2$. Meaning, $\mathfrak{x}$ is a nonzero socle element of highest degree in the former graded $\CS$-module and likewise for $\mathfrak{z}$ in the latter module. Given this, note that $(\mathfrak{x},0)\in\Tateext_R^{s_1}(M\oplus N, \k)$ and $(0,\mathfrak{z})\in\Tateext_R^{s_2}(M\oplus N, \k)$ must both be annihilated by the maximal ideal $\x\subseteq\CS$ and so, $s\geq \max\{s_1,s_2\}$. Conversely, we know there exists some $0\neq(\mathfrak{x'},\mathfrak{z'})\in\soc(\gTateext{R}{M\oplus N}{\k})$ with $\mathfrak{x'}\in\Tateext_R^s(M,\k)$ and $\mathfrak{z'}\in\Tateext_R^s(N,\k)$ such that $(\mathfrak{x'},\mathfrak{z'})\in(0:_{\hat{E}} \x)$. Therefore, either $\mathfrak{x'}$ or $\mathfrak{z'}$ (or both) must be nonzero and annihilated by $\x$, implying that $s=\max\{s_1,s_2\}$. This gives proof of the first equality.
	
	To show the second equality, we give an appropriate analogue to the argument for critical degree. First note that the existence of nonzero elements $\bar{\mathfrak{x}}\in \Tateext_{R}^{\leq t_1}(M,\k)/\x\Tateext_{R}^{\leq t_1}(M,\k)$ and $\bar{\mathfrak{z}}\in \Tateext_{R}^{\leq t_2}(N,\k)/\x\Tateext_{R}^{\leq t_2}(N,\k)$ implies the existence of nonzero elements $(\bar{\mathfrak{x}},0)\in \Tateext_{R}^{\leq t_1}(M\oplus N,\k)/\x\Tateext_{R}^{\leq t_1}(M\oplus N,\k)$ and $(0,\bar{\mathfrak{z}})\in \Tateext_{R}^{\leq t_2}(M\oplus N,\k)/\x\Tateext_{R}^{\leq t_2}(M\oplus N,\k)$. Hence, $t\leq t_1$ and $t\leq t_2$, so that $t\leq\min\{t_1,t_2\}$. Conversely, the existence of a nonzero element $(\bar{\mathfrak{x'}},\bar{\mathfrak{z'}})\in\Tateext_{R}^{\leq t}(M\oplus N,\k)/\x\Tateext_{R}^{\leq t}(M\oplus N,\k)$ implies that at least $\bar{\mathfrak{x'}}\neq 0$ or $\bar{\mathfrak{z'}}\neq 0$, thereby proving the equality $t=\min\{t_1,t_2\}$.
\end{proof}
\begin{rmk}
	Notice that we did not deal with the case of infinite critical or cocritical degrees. Recall that whenever $R$ is a complete intersection the critical degree of any $R$-complex (and $R$-module) will be positively finite and, likewise, the cocritical degree will always be negatively finite. If at least one of $\comp{C}$ or $\comp{D}$ is periodic, then the given statements (and arguments) still apply.
\end{rmk}
As $\Ktac{R}$ is a \textit{thick} subcategory of the homotopy category, $\CK(R)$, it contains all retracts. Thus, if a complex $\comp{E}$ in $\Ktac{R}$ can be written as a direct sum, say $\Csum{C}{D}$, then we want to consider the \crco of each summand $\comp{C}$ and $\comp{D}$. 
Whenever $\comp{E}$ is minimal, each summand  $\comp{C}$ and $\comp{D}$ must be minimal as well, and $\doug^{\comp{E}}=\doug^{\comp{C}}\oplus\doug^{\comp{D}}\subseteq \m\comp{C}\oplus\m\comp{D}$. 
For an endomorphism $\mu\!:\comp{E}\to\shiftq{q}{\comp{E}}$, note that we get \emph{four} induced maps (two on each summand): \begin{gather*}
\mu_1\!:\comp{C}\to\shiftq{q}{\comp{C}} \\
\mu_2\!:\comp{D}\to\shiftq{q}{\comp{C}}  \\
\mu_3\!:\comp{C}\to\shiftq{q}{\comp{D}} \\
\mu_4\!:\comp{D}\to\shiftq{q}{\comp{D}}. 
\end{gather*}
If $\crdeg{\mu}{R}{\comp{E}}=s$, then there is some  $\mu_{n+q}\!:C_{n+q}\oplus D_{n+q}\to C_n\oplus D_n$ that is surjective for all $n>s$. However, note that the surjectivity onto one summand, take $C_n$ for example, may not occur from $\mu_{1,(n+q)}$ alone as the map $\mu_{2,(n+q)}$ theoretically could contribute in part to the surjectivity. Likewise, if $\cocrdeg{\mu}{R}{\comp{E}}=t$, the same could occur with injectivity for $n<t$. In either case, it is difficult to deduce what the \crco are on each summand from Definitions \ref{Def:crdeg} and \ref{Def:cocrdeg} alone. Hence, we again employ Propositions \ref{Prop:cr-cohom} and \ref{Prop:cocr-cohom} to understand the critical and cocritical degrees of retracts in $\Ktac{R}$.

\begin{prop}\label{Prop:retracts-CrCo}
	Let $R=Q/(\f)$, with $\f=f_1,\dots, f_c$ a regular $Q$-sequence, and further suppose $\comp{E}\in\Ktac{R}$ with decomposition $\comp{E}=\Csum{C}{D}$ (neither $\comp{C}$ nor $\comp{D}$ contractible). If $\crdeg{}{R}{\comp{E}}=s$ then $\crdeg{}{R}{\comp{C}}\leq s$ and $\crdeg{}{R}{\comp{D}}\leq s$, with at least one being an equality. Likewise, if $\cocrdeg{}{R}{\comp{E}}=t$ then $\cocrdeg{}{R}{\comp{C}}\geq t$ and $\cocrdeg{}{R}{\comp{D}}\geq t$, with at least one being an equality.
\end{prop}
\begin{proof}
	For simplicity, assume $\comp{E}$ is minimal and denote by $M\oplus N=\Im\doug^{\comp{E}}_0$ 
	 where $M=\im\doug^{\comp{C}}_0$ and $N=\im\doug^{\comp{D}}_0$. Proposition \ref{Prop:cr-cohom} implies that $s$ is the maximal degree of a nonzero element $(\mathfrak{x},\mathfrak{z})\in\hat{E}=\gTateext{R}{M}{\k}\oplus\gTateext{R}{N}{\k} \cong\gTateext{R}{M\oplus N}{\k}$ with $(\mathfrak{x},\mathfrak{z})\in(0:_{\hat{E}}\x)$. Hence, $\mathfrak{x}\in\hat{E}_M=\Tateext_{R}^s(M,\k)$ and $\mathfrak{z}\in\hat{E}_N=\Tateext_{R}^s(N,\k)$ (with at least one nonzero) such that $\mathfrak{x}\in(0:_{\hat{E}_M}\x)$ and $\mathfrak{z}\in(0:_{\hat{E}_N}\x)$. This implies either $\crdeg{}{R}{\comp{C}}\geq s$ or $\crdeg{}{R}{\comp{D}}\geq s$ (or both, in the case that $\mathfrak{x}\neq 0 \neq \mathfrak{z}$). Now suppose $\crdeg{}{R}{\comp{C}}=s'\gneq s$ so that there exists some nonzero element $\mathfrak{x'}\in\Tateext_{R}^{s'}(M,\k)\subset\Tateext_{R}^{>s}(M,\k)$ such that $\mathfrak{x'}\x=0$. Of course, this would then imply the element $(\mathfrak{x'},0)\in\Tateext_{R}^{s'}(M\oplus N, \k)$ is annihilated by $\x$ thus contradicting $s$ as the highest degree socle element. The same argument can be applied to $\gTateext{R}{N}{\k}$, so that we see both $\crdeg{}{R}{\comp{C}}$ and $\crdeg{}{R}{\comp{D}}$ must be \emph{bounded above} by $s$. Lastly, it is straightforward to see that \emph{at least one} of $\crdeg{}{R}{\comp{C}}$ or $\crdeg{}{R}{\comp{D}}$ must be \emph{equal to} $s$. 
	
	For cocritical degree, there must exist a lowest degree nonzero element $(\mathfrak{\bar{x}},\mathfrak{\bar{z}})\in\cosoc(\gTateext{R}{M\oplus N}{\k})$ with $\mathfrak{x}\in \widehat{\Ext}^{t}(M,\k)$ and $\mathfrak{z}\in \widehat{\Ext}^{t}(N,\k)$ (at least one nonzero). That is, there exists some nonzero element of degree $t$ such that $(\mathfrak{x},\mathfrak{z})\not\in\x \widehat{\Ext}^{\leq t}(M\oplus N,\k)$. Suppose that $\cocrdeg{}{R}{\comp{C}}=t'\lneq t$ so that there exists some $0\neq\mathfrak{x'}\not\in\x \widehat{\Ext}^{\leq t'}(M,\k)$, implying $(\mathfrak{x'},0)\not\in\x \widehat{\Ext}^{\leq t'}(M\oplus N,\k)$ and contradicting the assumption that $t$ is the lowest degree of such an element. The same argument can be applied to $\cocrdeg{}{R}{\comp{D}}$, so we see that the cocritical degrees of $\comp{C}$ and $\comp{D}$ are bounded below by $t$. On the other hand, since there exists some nonzero element of degree $t$ such that $(\mathfrak{x},\mathfrak{z})\not\in\x \widehat{\Ext}^{\leq t}(M\oplus N,\k)$, either $0\neq\mathfrak{x}\not\in\x \widehat{\Ext}^{\leq t}(M,\k)$ or $0\neq\mathfrak{z}\not\in\x \widehat{\Ext}^{\leq t}(N,\k)$, demonstrating that we must have equality for at least one of the cocritical degrees.
\end{proof}

\begin{example}
	Let $R=\k[\![x_1,\dots,x_n]\!]/\f$ where $\f=(f_1, \dots, f_c)$ in $\m=(x_1, \dots, x_n)$. Furthermore, denote by $\comp{K}$ the (minimal) totally acyclic complex in the complete resolution of $M=\k$. By  \cite[\S 2]{Eisenbud}, $\crdeg{}{R}{\k}=-1$. Now denote by $\comp{C}=\compK\oplus\shiftq{5}{\comp{K}}$ and note that Propositions \ref{Prop:crco-shifts} and \ref{Prop:sums-crco} tell us that $\crdeg{}{R}{\comp{C}}=\max\{0,5\}=5$. Similarly, we expect $\cocrdeg{}{R}{\comp{C}}=\min\{-1,4\}=-1$.
\end{example}
The previous example demonstrates that the lowest degree Betti number may not occur at homological degree $0$, or even at the critical degree. Rather, it could occur \emph{at} the cocritical degree, and is always guaranteed to occur \emph{between} the \crco (inclusive). Hence, this motivates questions about the distance between the \crco of an $R$-complex.

\subsection{Critical Diameter}\label{width}
Although the authors of \cite{homalgCI-codim2} provide a sufficient bound for $\crdeg{}{R}{M}$ when $\cx_RM=2$, such bounds are unknown for greater complexity. Moreover, the bound given is dependent upon the Betti sequence of the given module. Example 7.5 from \cite{cid} demonstrates an obstacle when trying to find appropriate bounds for $\crdeg{}{R}{M}$ over a given ring or modules of the same complexity: taking a higher cosyzygy (or syzygy) module will always yield a higher (resp. lower) critical degree. We provide the following example to demonstrate the same issue, even when we transition to the definitions in $\Ktac{R}$.
\begin{example}
	Denote by $\comp{C}\in\Ktac{R}$ the minimal $R$-complex with $\Im\doug^{\comp{C}}_0=M$ and $\comp{D}\in\Ktac{R}$ the minimal $R$-complex with $\Im\doug^{\comp{D}}_0=N$, so that there exist complete resolutions $\cres{\comp{C}}{\res{F}}{M}$ and $\cres{\comp{D}}{\res{G}}{N}$. Now suppose $M$ and $N$ are syzygies of each other, say $N=\Omega^{-n}M$ and $M=\Omega^nN$. Then note that $\comp{D}\simeq\shiftq{n}{\comp{C}}$ (in fact, they are isomorphic!). If $\crdeg{}{R}{\comp{C}}=s$ and $\cocrdeg{}{R}{\comp{C}}=t$, then $\crdeg{}{R}{\comp{D}}=s+n$ and $\cocrdeg{}{R}{\comp{D}}=t+n$ by Proposition \ref{Prop:crco-shifts}. However, suppose we instead start with the assumption that $\crdeg{}{R}{\comp{D}}=s'$ and $\cocrdeg{}{R}{\comp{D}}=t'$, viewing $\comp{C}=\shiftq{-n}{\comp{D}}$. Once again applying Proposition \ref{Prop:crco-shifts}, it is not difficult to see that $\crdeg{}{R}{\comp{C}}=s'-n$ and $\cocrdeg{}{R}{\comp{C}}=t'-n$. Note that in doing so, $s'-n=s$ and $t'-n=t$. \hfill$\diamond$
\end{example}
In the example above, note that while the \crco change under translations, the \emph{difference} between these two degrees does not alter: $$\crdeg{}{R}{\comp{C}}-\cocrdeg{}{R}{\comp{C}}=s-t=(s+n)-(t+n)=\crdeg{}{R}{\comp{D}}-\cocrdeg{}{R}{\comp{D}}.$$

\begin{defn}\label{def: complex-diameter}
	Let $\comp{C}$ be a totally acyclic $R$-complex with minimal subcomplex $\mincomp{C}$ 
	such that $\cx_R\mincomp{C}>1$. Then the $R$-\emph{diameter} of $\comp{C}$ is the distance between the critical and cocritical degrees of $\comp{C}$,
	$$\cdiam{R}{C}=\crdeg{}{R}{\comp{C}}-\cocrdeg{}{R}{\comp{C}}.$$
	Define $\cdiam{R}{C}=-\infty$ for any $\comp{C}$ with $\cx_Q\mincomp{C}=1$ or if $\comp{C}\simeq 0$.
\end{defn}
\begin{rmk}
	Note that since $R$ is a complete intersection, $\cdiam{R}{C}<\infty$ but, relaxing to a local ring $Q$ (not regular!), $\cdiam{Q}{C}=\infty$ if and only if $\crdeg{}{Q}{\comp{C}}=\infty$ or $\cocrdeg{}{Q}{\comp{C}}=-\infty$.
\end{rmk}
In the above definition, 
it should be clear that 	$$\cdiam{R}{C}\leq\inf\{\crdeg{\mu}{R}{\comp{C}}-\cocrdeg{\mu}{R}{\comp{C}} \: : \:\mu\in\End_{\CK(R)}(\comp{C})\}$$ with equality whenever $\mu$ realizes both the critical and cocritical degrees. 
\begin{prop}
	If $\comp{C}\simeq\comp{D}$, then $\cdiam{R}{C}=\cdiam{R}{D}$.
\end{prop}
\begin{proof}
	By Proposition \ref{Prop:hom-stable}, we have that $\crdeg{}{R}{\comp{C}}=\crdeg{}{R}{\comp{D}}$ and $\cocrdeg{}{R}{\comp{C}}=\cocrdeg{}{R}{\comp{D}}$. Thus, the statement follows directly from these observations.
\end{proof}

\begin{prop}
	For non-periodic complexes $\comp{C}$ and $\comp{D}$ in $\Ktac{R}$, it holds that: \begin{enumerate}[i.]
		\item $\cdiam{R}{C^*}=\cdiam{R}{C}-q$
		\item $\cdiam{R}{\Csum{C}{D}}=\max\{\crdeg{}{R}{\comp{C}},\crdeg{}{R}{\comp{D}}\}-\min\{\cocrdeg{}{R}{\comp{C}},\cocrdeg{}{R}{\comp{D}}\}$
	\end{enumerate} where $q=\deg(\mu)$ and $\mu$ the endomorphism realizing $\crdeg{}{Q}{\comp{C}}$.
\end{prop}
\begin{proof}
	This is a straightforward application of Propositions \ref{Prop:dual-crco} and \ref{Prop:sums-crco}.
\end{proof}
\begin{defn}
	Let $M$ be a finitely-generated $R$-module with $\cx_RM>1$. Then the \emph{critical diameter} of $M$ is defined to be $$\mdiam{R}{M}=\cdiam{R}{C}$$ where $\comp{C}\to\res{F}\to{M}$ is the minimal complete resolution of $M$. 
	Furthermore, define $\mdiam{R}{M}=-\infty$ for any $R$-module with $\cx_RM=1$ and set $\mdiam{R}{M}=0$ if $\pd_RM<\infty$.
\end{defn}

\begin{example}
		Let $Q=\k[\![x,y]\!]$ and take $R=Q/(x^2,y^2)$ so $\codim(Q,R)=2$ and we may expect two CI operators. Now, suppose $M=\k=\coker(f)$ where $f=(x, y)$ from $R^2\to R$ and take a minimal complete resolution $\comp{K}\to\res{F}\to\k$. Then the CI (Eisenbud) operators on $\comp{K}$ form degree $-2$ maps on $\comp{K}$ as follows as follows: \tiny	
 			
 	$$\xymatrix @R=.5cm @C=1.2cm{
 		R^{4}\ar[r]^{\begin{psmallmatrix} y & x & 0 & 0 \\ 0 & -y & x & 0 \\ 0 & 0 & y & x \end{psmallmatrix}} \ar[d]_{t_3} & R^{3}\ar[r]^{\begin{psmallmatrix} y & x & 0 \\ 0 & -y & x \end{psmallmatrix}} \ar[d]_{t_2} & R^{2}\ar[r]^{\begin{psmallmatrix}
 			x & y \end{psmallmatrix}} \ar[d]_{t_1}  & R\ar[r]^{\begin{psmallmatrix}
 			xy	\end{psmallmatrix}} \ar[d]_{t_0} & R\ar[r]^{\begin{psmallmatrix}
 			x \\ y \end{psmallmatrix}} \ar[d]_{t_{-1}}  &R^{2}\ar[r]^{\begin{psmallmatrix} y & 0 \\ -x & y \\ 0 & x \end{psmallmatrix}} \ar[d]_{t_{-2}}  & R^3\ar[r]^{\begin{psmallmatrix} y & 0 & 0 \\ x & y & 0 \\ 0 & -x & y \\ 0 & 0 & x \end{psmallmatrix}} \ar[d]_{t_{-3}}  & R^4 \ar[d]_{t_{-3}} 
 			\\ R^{2}\ar[r] & R\ar[r] & R\ar[r] & R^2 \ar[r] & R^3 \ar[r] &R^{4}\ar[r] &R^5 \ar[r] & R^6 \\}$$
 	 	\vspace{-4mm}$$\xymatrix @R=.05cm @C=0.95cm{\color{gray}\text{Deg: } (1) &\color{gray} \hspace{-2mm}(0) & \hspace{1mm}\color{gray}(-1) & \hspace{1mm} \color{gray}(-2) & \color{gray}(-3) &\color{gray} (-4) & \color{gray}(-5) & \color{gray}(-6)  }$$
 	 	
	\normalsize where \begin{itemize}
		\item 	$t_{x,3} = \begin{psmallmatrix} 0 & 0 & 1 & 0 \\ 0 & 0 & 0 & 1 \end{psmallmatrix}$ and $t_{y,3} = \begin{psmallmatrix} 1 & 0 & 0 & 0 \\ 0 & 1 & 0 & 0 \end{psmallmatrix} $,\vspace{2mm}
		
		\item $t_{x,2} = \begin{psmallmatrix} 0 & 0 & 1 \end{psmallmatrix} $ and	$t_{y,2} = \begin{psmallmatrix} 1 & 0 & 0 \end{psmallmatrix} $,\vspace{2mm}
		
		\item $t_{x,1} = \begin{psmallmatrix} 0 & y \end{psmallmatrix}$ and $t_{y,1} = \begin{psmallmatrix} x & 0 \end{psmallmatrix}$,\vspace{2mm}
		
		\item $t_{x,0} = \begin{psmallmatrix}	0 \\ y \end{psmallmatrix}$ and $t_{y,0} = \begin{psmallmatrix}	x \\ 0 \end{psmallmatrix}$,\vspace{2mm}
		
		\item $t_{x,-1} =\begin{psmallmatrix}	0 \\ 0 \\ 1	\end{psmallmatrix}$ and $t_{y,-1} =\begin{psmallmatrix}	1 \\ 0 \\ 0	\end{psmallmatrix}$,\vspace{2mm}
		
		\item $t_{x,-2}= \begin{psmallmatrix} 0 & 0 \\ 0 &0 \\ 1 & 0 \\ 0 & 1	\end{psmallmatrix}$ and $t_{y,-2}= \begin{psmallmatrix} 1 & 0 \\ 0 &1 \\ 0 & 0 \\ 0 & 0	\end{psmallmatrix}$,\vspace{2mm}
		
		\item $t_{x,-3}= \begin{psmallmatrix} 0 & 0 & 0\\ 0 &0 & 0\\ 1 & 0 &0 \\ 0 & 1 & 0 \\ 0 &0 & 1	\end{psmallmatrix}$ and $t_{y,-3}= \begin{psmallmatrix} 1 & 0 &0\\ 0 &1 &0\\ 0 & 0 & 1 \\ 0 & 0 & 0 \\0 & 0& 0	\end{psmallmatrix}$,\vspace{2mm}
		
		\item $t_{x,-4}= \begin{psmallmatrix} 0 & 0 & 0 & 0 \\0 & 0 &0 & 0\\ 1 & 0 &0 & 0\\ 0 & 1 & 0 & 0\\ 0 &0 & 1 & 0 \\ 0 & 0 & 0 & 1	\end{psmallmatrix}$ and $t_{y,-4}= \begin{psmallmatrix} 1 & 0 &0 & 0\\ 0 &1 &0 & 0\\ 0 & 0 & 1 & 0\\ 0 & 0 & 0 & 1 \\ 0 & 0& 0 & 0 \\ 0 & 0& 0 & 0	\end{psmallmatrix}$
	\end{itemize}
	As Eisenbud demonstrated in \cite[\S2]{Eisenbud}, each $t_{i+2}$ is surjective for $i\geq 0$ and thus $\crdeg{}{R}{C}=-1$ in this case. Interestingly, $\cocrdeg{}{R}{C}=-1$ as well since we see that the earliest injective map on the right is $t_{-1}$. This indicates $\cdiam{R}{K}=-1-(-1)=0=\mdiam{R}{\k}$. In general, we conjecture that the residue field over any complete intersection ring will have a critical diameter of $0$.
\end{example}

\subsection{Boundedness Problems on Diameter}
Note first that while $\crdeg{}{R}{M}\neq\crdeg{}{R}{\Syz{n}{M}}$ for any $n\in\Z^{+}\cup\Z^{-}$, it should be clear that $\mdiam{R}{M}=\mdiam{R}{\Syz{n}{M}}$. However, given two $R$-modules $N$ and $M$ with $2\leq\cx_RM=\cx_RN$, does there exist a common upper bound for $\mdiam{R}{M}$ and $\mdiam{R}{N}$? We now shift perspective to an $R$-complex to partly address this question.
\begin{example}
	Let $M$ be a finitely-generated MCM $R$-module with minimal complete resolution $\cres{C}{F}{M}$ and suppose $\cx_RM>1$. Denote by $\mdiam{R}{M}=\omega_M=\cdiam{R}{C}$. Consider the $R$-complex $\comp{C}\oplus\shiftq{n}{\comp{C}}$ which is associated to the (minimal) complete resolution $\comp{C}\oplus\shiftq{n}{\comp{C}}\to\res{F}\oplus\shiftq{n}{\res{F}}\twoheadrightarrow M\oplus\Omega^nM$ for some fixed integer $n\in\Z$. Hence, $\cx_R(M\oplus\Omega^nM)=\cx_RM$ and yet $\omega_{M\oplus\Omega^nM}<\omega_{M\oplus\Omega^{n+1}M}$ for any such integer $n\in\Z$.
\end{example}
This demonstrates that the diameter is not necessarily bounded for all modules (or complexes) of a given complexity (greater than one). We leave the reader with the same question for an \emph{indecomposable} MCM module (complex).
\begin{OProb}
	Let $R$ be a complete intersection ring and denote by $M$ a finitely-generated MCM $R$-module with minimal complete resolution $\cres{C}{F}{M}$ \emph{such that $\comp{C}$ is indecomposable}. Further suppose that $\cx_RM>1$ and denote by $\mdiam{R}{M}=\omega_M=\cdiam{R}{C}$. Then does there exist some $d\in\N$ such that $\omega_M\leq d$ for all such finitely-generated MCM $R$-modules $M$ with indecomposable complete resolution and the same complexity (greater than one)? 
\end{OProb}

\subsection*{Acknowledgments} I would like to sincerely thank my advisor David Jorgensen, who provided direction for this work in addition to having numerous discussions about it. I would also like to thank Tyler Anway for many, \emph{many} conversations and comments throughout the writing process.


\bibliographystyle{acm}
\bibliography{CrCoDeg}

\begin{thebibliography}{10}

\bibitem{RcatMods}
{\sc Anderson, F.~W., and Fuller, K.~R.}
\newblock {\em Rings and Categories of Modules}.
\newblock Graduate Texts in Mathematics. Springer New York, 2012.

\bibitem{buch-aus}
{\sc Auslander, M., and Buchsbaum, D.~A.}
\newblock Homological dimension in local rings.
\newblock {\em Transactions of the American Mathematical Society 85}, 2 (1957),
  390--405.

\bibitem{Avramov}
{\sc Avramov, L.~L.}
\newblock Modules of finite virtual projective dimension.
\newblock {\em Inventiones mathematicae 96}, 1 (1989), 71--102.

\bibitem{homalgCI-codim2}
{\sc Avramov, L.~L., and Buchweitz, R.-O.}
\newblock Homological algebra modulo a regular sequence with special attention
  to codimension two.
\newblock {\em Journal of Algebra 230}, 1 (2000), 24--67.

\bibitem{cid}
{\sc Avramov, L.~L., Gasharov, V.~N., and Peeva, I.~V.}
\newblock Complete intersection dimension.
\newblock {\em Publications Math\'ematiques de l'IH\'ES 86\/} (1997), 67--114.

\bibitem{rel}
{\sc Avramov, L.~L., and Martsinkovsky, A.}
\newblock Absolute, relative, and tate cohomology of modules of finite
  gorenstein dimension.
\newblock {\em Proceedings of the London Mathematical Society 85}, 2 (2002),
  393–440.

\bibitem{CMR_Bru-Her}
{\sc Bruns, W., and Herzog, J.}
\newblock {\em Cohen-Macaulay Rings}, 2~ed.
\newblock Cambridge Studies in Advanced Mathematics. Cambridge University
  Press, 1998.

\bibitem{exact-cat-dual_Buch}
{\sc Buchsbaum, D.~A.}
\newblock Exact categories and duality.
\newblock {\em Transactions of the American Mathematical Society 80}, 1 (1955),
  1--34.

\bibitem{mcm-tate}
{\sc Buchweitz, R.-O.}
\newblock Maximal cohen–macaulay modules and tate cohomology over gorenstein
  rings.

\bibitem{pten}
{\sc Christensen, L.~W., and Jorgensen, D.~A.}
\newblock Tate (co)homology via pinched complexes.
\newblock {\em Transactions of the American Mathematical Society 366}, 2
  (2014), 667--689.

\bibitem{Eisenbud}
{\sc Eisenbud, D.}
\newblock Homological algebra on a complete intersection, with an application
  to group representations.
\newblock {\em Transactions of the American Mathematical Society 260}, 1
  (1979), 35--64.

\bibitem{eisenbud1995comm}
{\sc Eisenbud, D.}
\newblock {\em Commutative Algebra: With a View Toward Algebraic Geometry}.
\newblock Graduate Texts in Mathematics. Springer, 1995.

\bibitem{minfreeresCI}
{\sc Eisenbud, D., and Peeva, I.~V.}
\newblock {\em Minimal Free Resolutions over Complete Intersections},
  vol.~2152.
\newblock 01 2016.

\bibitem{gu}
{\sc Gulliksen, T.~H.}
\newblock A change of ring theorem with applications to poincaré series and
  intersection multiplicity.
\newblock {\em Mathematica Scandinavica 34\/} (1974), 167--183.

\bibitem{coreg-quasi}
{\sc Hartshorne, R., and Polini, C.}
\newblock Quasi-cyclic modules and coregular sequences.
\newblock {\em Mathematische Zeitschrift\/} (2021).

\bibitem{tricat}
{\sc Holm, T., J{\o}rgensen, P., and Rouqui\''er, R.}
\newblock {\em Triangulated Categories}.
\newblock London Mathematical Society Lecture Note Series. Cambridge University
  Press, 2010.

\bibitem{mac2013categories}
{\sc Lane, S.~M.}
\newblock {\em Categories for the working mathematician}, vol.~5.
\newblock Springer Science \& Business Media, 2013.

\bibitem{DCC-mat}
{\sc Matlis, E.}
\newblock Modules with descending chain condition.
\newblock {\em Transactions of the American Mathematical Society 97}, 3 (1960),
  495--508.

\bibitem{neeTri}
{\sc Neeman, A.}
\newblock {\em Triangulated Categories}.
\newblock Academic Search Complete. Princeton University Press, 2001.

\bibitem{Oomat-dual}
{\sc Ooishi, A.}
\newblock Matlis duality and the width of a module.
\newblock {\em Hiroshima Mathematical Journal 6}, 3 (1976), 573 -- 587.

\bibitem{taa}
{\sc Petter A.~Bergh, P., Jorgensen, D.~A., and Moore, F.~W.}
\newblock Totally acyclic approximations.
\newblock {\em Applied Categorical Structures 29\/} (08 2021).

\bibitem{rotman-homAlg}
{\sc Rotman, J.~J.}
\newblock {\em An Introduction to Homological Algebra}.
\newblock No.~v. 85 in An introduction to homological algebra. Academic Press,
  1979.

\bibitem{tate}
{\sc Tate, J.~T.}
\newblock Homology of noetherian rings and local rings.
\newblock {\em Illinois Journal of Mathematics 1\/} (1957), 14--27.

\bibitem{gorproj-comp}
{\sc Veliche, O.}
\newblock Gorenstein projective dimension for complexes.
\newblock {\em Transactions of the American Mathematical Society 358}, 3
  (2006), 1257--1283.

\end{thebibliography}
\nocite{*}


\end{document}